\documentclass[preprint,1p,10pt]{elsarticle}

\journal{Arxiv}

\biboptions{sort&compress,comma,round}
\usepackage{stmaryrd}
\usepackage{graphicx}
\usepackage{amssymb}
\usepackage{bm}

\usepackage{amsmath}
\usepackage{float}
\DeclareSymbolFont{largesymbolsA}{U}{txexa}{m}{n}
\usepackage{setspace}
\usepackage{graphics}
\usepackage{amssymb, latexsym, amsmath, epsfig}
\usepackage{graphicx}
\usepackage{epstopdf}
\usepackage{comment}
\usepackage[normalem]{ulem}
\usepackage{float}
\usepackage{amssymb,amsmath}
\usepackage{amsfonts,amstext,amsthm,amssymb,comment}
\usepackage{epsfig,graphics,color}

\usepackage{lineno}
\usepackage{amsmath}
\usepackage{amsthm}
\usepackage{amssymb}
\usepackage{color}

\usepackage{enumitem}
\usepackage{stmaryrd}
\usepackage{algpseudocode} 
\usepackage{algorithm}
\usepackage{amssymb,amsmath}
\usepackage{textcomp}
\usepackage{stmaryrd}
\modulolinenumbers[5]

\usepackage{float}

\usepackage{graphicx}
\usepackage{epstopdf}
\usepackage{subcaption} 

\newcommand{\bfs}[1]{{\boldsymbol #1}}
\newtheorem{thm}{Theorem}

\newtheorem{ass}{Assumption}
\newtheorem{remark}{Remark}

\newcommand{\jac}[1]{\ensuremath{\det(D #1)}}
\newcommand{\del}[1]{\ensuremath{\llbracket #1 \rrbracket}}

\biboptions{square,sort,compress,comma,numbers}
\def \M{\mathbf{M}}
\def \D{\mathbf{D}}
\def \P{\boldsymbol{\mathcal{M}}}
\def \Ndof{N_{\text{dof}}}
\def \PMP{\P^{-\frac{1}{2}}\M\P^{-\frac{1}{2}}}

\def \Mp{\widehat{\M}}

\def \vett[#1]{\boldsymbol{#1}}

\def \supp{\mathrm{supp}}

\def \Prs{\D^{\frac{1}{2}} \widehat{\D}^{-\frac{1}{2}} \Mp \widehat{\D}^{-\frac{1}{2}} \D^{\frac{1}{2}}}

\def \spann{\mathrm{span}}

\def \NOmega{N_{\mathrm{patch}}}
\def \Npatch{N_{\mathrm{adj}}}
\def \ptca{r}
\def \ptcb{s}
\def \Prsa{\D^{(\ptca)^{\frac{1}{2}}} \widehat{\D}^{(\ptca)^{-\frac{1}{2}}} \Mp^{(\ptca)} \widehat{\D}^{(\ptca)^{-\frac{1}{2}}} \D^{(\ptca)^{\frac{1}{2}}}}

\def \Prad{\P^{-1}_{\text{ad}}}
\def \Bmp{\mathcal{B}_{\text{mp}}}
\def \PadMPad{\P^{-\frac{1}{2}}_{\text{ad}}\M \P^{-\frac{1}{2}}_{\text{ad}}}

\begin{document}
	\begin{frontmatter}
		
		
		\title{Explicit high-order generalized-$\alpha$ methods for isogeometric analysis of structural dynamics}
		
		% \author[add]{Ivo Babu\v{s}ka}
		% 
		\author[ad1]{Pouria Behnoudfar\corref{corr}}
		\ead{pouria.behnoudfar@postgrad.curtin.edu.au}
		\cortext[corr]{Corresponding author}
		\author[ad2]{Gabriele Loli}
		\author[ad3]{Alessandro Reali}
		\author[ad2]{Giancarlo Sangalli}
		\author[ad4]{Victor M. Calo}		
		
		\address[ad1]{Curtin Institute for Computation \& School of Earth and Planetary Sciences, Curtin University, Kent Street, Bentley, Perth, WA 6102, Australia}
		\address[ad2]{Dipartimento di Matematica ``F. Casorati", Universit\`{a} di Pavia, 27100 Pavia, Italy}
		\address[ad3]{Department of Civil Engineering and Architecture, Universit\`{a} di Pavia, 27100 Pavia, Italy}
		\address[ad4]{School of Electrical
			Engineering, Computing and Mathematical Sciences, Curtin
			University, P.O. Box U1987, Perth, WA 6845, Australia}
		% 
		% \address[ad3]{Department of Mathematical and Statistical Sciences, University of Alberta, Edmonton, Alberta, Canada T6G 2G1}
		
		\begin{abstract}
			
			We propose a new family of high-order explicit generalized-$\alpha$ methods for hyperbolic problems with the feature of dissipation control. Our approach delivers $2k,\, \left(k \in \mathbb{N}\right)$ accuracy order in time by solving $k$ matrix systems explicitly and updating the other $2k$ variables at each time-step. The user can control the numerical dissipation in the discrete spectrum's high-frequency regions by adjusting the method's coefficients. We study the method's spectrum behaviour and show that the CFL condition is independent of the accuracy order. The stability region remains invariant while we increase the accuracy order. Next, we exploit efficient preconditioners for the isogeometric matrix to minimize the computational cost. These preconditioners use a diagonal-scaled Kronecker product of univariate parametric mass matrices; they have a robust performance with respect to the spline degree and the mesh size, and their decomposition structure implies that their application is faster than a matrix-vector product involving the fully-assembled mass matrix. Our high-order schemes require simple modifications of the available implementations of the generalized-$\alpha$ method. Finally, we present numerical examples demonstrating the methodology's performance regarding single- and multi-patch IGA discretizations. 
		\end{abstract}
		
		
	\end{frontmatter}
	
	\section{Introduction}
	
	Chung~and~Hulbert~\cite{chung1993time} introduced the generalized-$\alpha$ method for hyperbolic systems arising in structural dynamics; later, Hulbert~and~Chung~\cite{ HULBERT1996175} presented an explicit version of this time-marching method. The method is second-order accurate in time and provides user-control on the high-frequency numerical dissipations. Likewise, other well-known methods like, e.g., the Newmark-$\beta$~\cite{ newmark1959method} and the HHT-$\alpha$~\cite{ hilber1977improved} methods have both implicit and explicit formulations and are limited to the second-order accuracy. Nevertheless, the Newmark method does not control the dissipation while, using the HHT-$\alpha$, one obtains dissipative solutions in low-frequency regions. The generalized-$\alpha$ method produces an algorithm that combines high-frequency and low-frequency dissipations optimally; see~\cite{ chung1993time, HULBERT1996175}.
	
	These methods, including the generalized-$\alpha$ method, are limited to the second-order accuracy in time. In contrast, high-order methods like Lax-Wendroff, Runge-Kutta, Adams-Moulton, and backward differentiation schemes (see~\cite{ butcher2016numerical}) or high-order IGA collocation methods for dynamics (see, e.g.,~\cite{evans2018explicit})
	lack explicit control over the numerical dissipation of high frequencies. Thus, we propose an explicit $k$-step generalized-$\alpha$ method that delivers $2k$ accuracy in time for second derivatives problems in time. Our method obtains high-order accuracy in the time with optimal control of the resulting system's spectral behaviour. Thus, we build an algorithm that consists of $3k$ equations; for each three-equation set, we solve an explicit system for a variable and update the other two. We then study the resulting amplification matrix's spectral properties to determine the CFL condition and introduce user-defined parameters that control the numerical dissipation. Later, we prove that the CFL condition of our method is independent of temporal accuracy.
	
	The explicit generalized-$\alpha$ method, at each time step,  factors the isogeometric mass matrix; thus, we use a preconditioned conjugate gradients (PCG) as an iterative solver using $\boldsymbol{\mathcal{M}}$~\cite{ loli2020}, which is easy to implement, extremely efficient and robust.  We build the preconditioner using a diagonal scaling of the parametric mass matrix (i.e., matrix associated with the $d$-dimensional cube pre-image of the parametric space) for single-patch geometries. We combine the preconditioners defined above for each patch for multi-patch geometries using an additive Schwarz domain-decomposition method. In~\cite{ loli2020}, the authors proved that these preconditioners are efficient to apply and are robust with respect to the mesh size; in the single-patch case, they show that the behaviour improves as the problem size grows (mesh refinement).
	
	We present numerical simulations to demonstrate the performance of the time-marching scheme and its preconditioner. We provide numerical evidence on the time integrator's high-order accuracy and its dispersion properties. Additionally, we show the optimal convergence in the spatial domain for several problems. The outline for the remainder of the paper is as follows. Section~\ref{sec:2} describes the hyperbolic problem we consider and introduces the spatial discretizations to obtain the matrix formulation of the problem. Section~\ref{sec:3} presents our fourth-order explicit generalized-$\alpha$ method; therein, we also analyze the method's stability, its CFL condition, and its temporal accuracy. Section~\ref{sec:4} generalizes the method to $2k^{th}$-order of accuracy. Then, we introduce our solver in Section~\ref{sec:5} with details on single- and multi-patch isogeometric analysis. We verify the solver's convergence and its computational performance in Section~\ref{sec:6} numerically. Section~\ref{sec:7} describes our contributions and further applications.
	
	\section{Problem Statement} \label{sec:2}
	
	We start with an initial boundary-value hyperbolic problem, a model problem for structural dynamics: 
	\begin{equation} \label{eq:pde}
	\begin{aligned}
	\begin{cases}
	\ddot{u}(\boldsymbol{x},t) - \nabla \cdot(\omega^2\, \nabla u(\boldsymbol{x}, t)) +c\left(\dot{u}(\boldsymbol{x}, t)\right)& = f(\boldsymbol{x},t), \qquad \,\, \,(\boldsymbol{x}, t) \in \Omega \times [0, T], \\
	u(\boldsymbol{x}, t) & = u_D, \qquad \qquad \boldsymbol{x} \in \partial \Omega, \\
	u(\boldsymbol{x}, 0) & = u_0, \qquad \,\qquad \boldsymbol{x} \in \Omega, \\
	\dot{u}(\boldsymbol{x},0) & = v_0, \qquad\,\, \qquad \boldsymbol{x} \in \Omega,
	\end{cases}
	\end{aligned}
	\end{equation}
	Let $\Omega \subset \mathbb{R}^d, d=1,2,3,$ be an open bounded domain. The operator $\nabla$ is the spatial gradient and a superscript dot denotes a time derivative such that $ \dot{u}(\boldsymbol{x},t)=\frac{\partial {u}(\boldsymbol{x},t)}{ \partial t}$ and $ \ddot{u}(\boldsymbol{x},t)=\frac{\partial^2 {u}(\boldsymbol{x},t)}{ \partial t^2}$. $c\left(\dot{u}(\boldsymbol{x}, t)\right)$ models linear damping. The source $f$, propagation speed $\omega$, initial data $u_0, \, v_0$, and Dirichlet boundary conditions $u_D$ are given and assumed regular enough for the problem to admit a weak solution. In order to derive our numerical method for~\eqref{eq:pde}, we first obtain a semi-discretized problem by discretizing in space, then, we deploy our explicit generalized-$\alpha$ method to have a fully discretized system.
	
	\subsection{Spatial discretization}
	
	Adopting a Galerkin method (in particular, isogeometric analysis), the matrix problem resulting from the semi-discretization of~\eqref{eq:pde} reads:
	\begin{equation} \label{eq:mp}
	M \ddot{U} +C\dot{U}+ KU = F,
	\end{equation}
	where $M$, $C$, and $K$ are the mass, damping, and stiffness matrices, respectively. $U$ denotes the vector of the unknowns, and $F$ is the source vector. The initial conditions also read:
	\begin{equation} \label{eq:u0}
	U(0) = U_0, \qquad V(0) = V_0, 
	\end{equation}
	where $U_0$ and $V_0$ represent the given vectors initial conditions corresponding to $u_{0,h}$ and $\dot u_{0,h}$, respectively. In the next section, we propose our numerical technique to deal with the time derivative $\ddot U$ and $\dot U$ in~\eqref{eq:mp} with the accuracy of order $2k$ in the temporal domain with $k\in \mathbb{N}$.
	
	\begin{remark}
		Herein, we propose an explicit generalized-$\alpha$ scheme by considering a general spatial discretization, leading to the matrix problem~\eqref{eq:mp}. Therefore, the use of isogeometric analysis does not limit the method's applicability; effectively, our method applies to the spatial discretization of a time-dependent semi-discretized problem.
	\end{remark}
	\begin{remark}
		In problem~\eqref{eq:pde} and accordingly~\eqref{eq:mp}, for simplicity, we only consider constant $\omega$ and assume that the solution $u(\cdot\,,t)$ satisfies homogeneous boundary condition. One requires slight modifications of the discrete bilinear and linear functions for the cases of heterogeneous propagation speed and non-homogeneous boundary conditions~\cite{hughes2014finite,  hughes2012finite}.
	\end{remark}

	\subsection{Time-discretization}
	
	To obtain a fully discrete problem~\eqref{eq:mp}, we adopt an appropriate time marching scheme to deal with $\ddot u_h$ and $\dot u_h$; in the next section, we propose a new high-order explicit generalized-$\alpha$ method.  
	
	\section {Explicit generalized-$\alpha$ method}\label{sec:3}
	
	Consider a partitioning of the time interval $[0,T]$ as $0 = t_0 < t_1 < \cdots < t_N = T$ with a grid size $\tau_n=t_{n+1}-t_n$. We  approximate $U(t_n),\, \dot U(t_n),\, \ddot U(t_n)$ using $U_n, \,V_n,\, A_n$, respectively. Tthe explicit generalized-$\alpha$ method with second-order accuracy in time solves~\eqref{eq:mp} at time-step $t_{n+1}$,
	\begin{subequations} \label{eq:galpha}
		\begin{align}
		M A_{n+\alpha_m} + CV_n+K U_{n} & = F_{n+\alpha_f},\label{eq:g1} \\
		V_{n+1} & = V_n + \tau A_n + \tau \gamma \del{ A_n},\label{eq:g2} \\
		U_{n+1} & = U_n + \tau V_n + \frac{\tau^2}{2} A_n+ \tau^2 \beta  \del{A_n} \label{eq:g3}, 
		\end{align}
	\end{subequations}
	where
	\begin{equation} \label{eq:mf}
	\begin{aligned}
	F_{n+\alpha_f} & = F(t_{n+\alpha_f}), \\
	A_{n+\alpha_m} & = A_n + \alpha_m \del{A_n}\\ 
	\del{A_n}& = A_{n+1} - A_n.
	\end{aligned}
	\end{equation}
	To retrieve the unknown $A$ at the initial state, we solve
	\begin{equation}\label{eq:ini}
	A_0  = M^{-1} (F_0 - K U_0-CV_0).
	\end{equation}
	At each time step, we first compute $A_{n+1}$ using~\eqref{eq:g1}, and then evaluate $V_{n+1}$ and $U_{n+1}$ from~\eqref{eq:g2} and~\eqref{eq:g3}, respectively. The method~\eqref{eq:galpha} has a truncation error in time of~$\mathcal{O}(\tau^3)$ (for further details, see~\cite{ deng2019high, behnoudfar2019higher}). We extend the accuracy, assuming sufficient temporal regularity of the solution; thus, we use a Taylor expansion with higher-order terms. We introduce new variables $\mathcal{L}^a(A_n)$ to approximate the $a$-th order derivative of $A_n$ in time. Therefore, for example, to derive a fourth-order explicit generalized-$\alpha$ method, we solve
	\begin{equation} \label{eq:4a}
	\begin{aligned}
	MA_{n}^{\alpha_1}+CV_n+KU_{n}&=F_{n+\alpha_{f1}}, \\
	M\mathcal{L}^{3}(A_n)^{\alpha_2}+C\mathcal{L}^{2}(A_n)+K\mathcal{L}^{1}(A_n)&=F^{(3)}_{n+\alpha_{f2}} ,
	\end{aligned}
	\end{equation}
	with updating conditions
	\begin{equation} \label{eq:4aup}
	\begin{aligned}
	U_{n+1}  &= U_n + \tau V_n + \frac{\tau^2}{2} A_n + \frac{\tau^3}{6} \mathcal{L}^{1}(A_n)+ \frac{\tau^4}{24}\mathcal{L}^{2}(A_n)+ \frac{\tau^5}{120} \mathcal{L}^{3}(A_n)+\beta_1 \tau^2 P_{n}, \\
	V_{n+1}  &= V_n + \tau A_n + \frac{\tau^2}{2} \mathcal{L}^{1}(A_n) + \frac{\tau^3}{6} \mathcal{L}^{2}(A_n)+ \frac{\tau^4}{24} \mathcal{L}^{3}(A_n)+\gamma_1 \tau P_{n}, \\[0.2cm]
	\mathcal{L}^{1}(A_{n+1})&= \mathcal{L}^{1}(A_{n}) + \tau \mathcal{L}^{2}(A_n) + \frac{\tau^2}{2}\mathcal{L}^{3}(A_n) + \tau^2 \beta_2  \del {\mathcal{L}^{3}(A_{n})}, \\
	\mathcal{L}^{2}(A_{n+1})&= \mathcal{L}^{2}(A_n) + \tau \mathcal{L}^{3}(A_n) + \tau \gamma_2 \del {\mathcal{L}^{3}(A_{n})} ,
	\end{aligned}
	\end{equation}
	where
	\begin{equation} 
	\begin{aligned}
	P_{n}&=A_{n+1}-A_n-\tau \mathcal{L}^{1}(A_n)-\frac{\tau^2}{2} \mathcal{L}^{2}(A_n)-\frac{\tau^3}{6} \mathcal{L}^{3}(A_n),\\
	A_{n}^{\alpha_1}&=A_n+\tau \mathcal{L}^{1}(A_{n})+\frac{\tau^2}{2} \mathcal{L}^{2}(A_{n})+\frac{\tau^3}{6} \mathcal{L}^{3}(A_{n})+\alpha_1 P_{n},\\[0.2cm]
	\mathcal{L}^{3}(A_n)^{\alpha_2}&=\mathcal{L}^{3}(A_n)+\alpha_2\del {\mathcal{L}^{3}(A_{n})}.
	\end{aligned}
	\end{equation}
	We approximate $\frac{\partial^{(a)}}{\partial t^{(a)}} f(\cdot\,,\, n+\alpha_{f2})$ as $ F^{(a)}_{n+\alpha_{f2}}$. Following~\eqref{eq:ini}, one can readily obtain the initial data of the unknowns using the given information on $U_0$ and $V_0$ as
	\begin{equation}
	\begin{aligned}
	A_0  &= M^{-1} \left(F_0 - K U_0-CV_0\right), \\
	\mathcal{L}^{1}(A_0)&=M^{-1} \left(F_0^{(1)} - K V_0-CA_0\right),\\ 
	\mathcal{L}^{2}(A_0)&=M^{-1} \left(F_0^{(2)} - K A_0-C\mathcal{L}^{1}(A_0)\right),\\
	\mathcal{L}^{3}(A_0)&=M^{-1} \left(F_0^{(3)} - K \mathcal{L}^{1}(A_0)-C\mathcal{L}^{2}(A_0)\right).
	\end{aligned}
	\end{equation}
	Next, we derive the corresponding coefficients that deliver the desired accuracy and discuss the method's stability.

	\subsection{ Order of accuracy in time}
	
	We now determine the parameters $\gamma_1$ and $\gamma_2$ that guarantee the fourth-order accuracy of~\eqref{eq:4a}; thus, we use the following result.
	\begin{thm} \label{thm:3o}
		Assuming that the solution is sufficiently smooth with respect to time, the method~\eqref{eq:4a} is fourth-order accurate in time given
		\begin{equation} \label{eq:3ov1}
		\gamma_1=\frac{1}{2}-\alpha_{f1}+\alpha_1, \qquad \qquad     \gamma_2=\frac{1}{2}-\alpha_{f2}+\alpha_2.
		\end{equation}
	\end{thm}
	\begin{proof}
		Following~\cite{ hairer2010solving}, we determine the amplification matrix by associating with $\lambda=\theta^2$  the eigenvalues of the matrix $M^{-1}K$. The damping term's eigenvalues $M^{-1}C$ are $\zeta \theta$ where $\zeta$ is the damping coefficient. Without loss of the proof's generality, we set $\zeta=0$ and ignore damping. Then, substituting~\eqref{eq:4aup} into~\eqref{eq:4a}, we obtain a system of equations at each time step as
		\begin{equation}
		A\bold{U}_{n+1}=B \bold{U}_n+\bold{F}_{n+\alpha_f},
		\end{equation}
		letting $      \bold{U}_{n}^T=
		\begin{bmatrix}
		U_{n},\
		\tau V_{n},\
		\tau^2 A_{n},\
		\tau^3\mathcal{L}^{1}(A_{n}),\
		\tau^4 \mathcal{L}^{2}(A_{n}),\
		\tau^5 \mathcal{L}^{3}(A_{n})
		\end{bmatrix}^T
		$,
		\begin{equation}
		\begin{aligned}A&=
		\begin{bmatrix}
		1 & 0 & -\beta_1&0&0&0 \\
		0 & 1 & -\gamma_1&0&0&0 \\
		0 & 0 & \alpha_1&0&0&0 \\
		0 & 0 & 0&1&0&-\beta_2 \\
		0 & 0 & 0&0&1&-\gamma_2 \\
		0 & 0 & 0&0&0&\alpha_2 
		\end{bmatrix},
		\\
		B&=
		\begin{bmatrix}
		1 & 1 & \frac{1}{2}-\beta_1& \frac{1}{6}-\beta_1& \frac{1}{24}-\frac{\beta_1}{2}& \frac{1}{120}-\frac{\beta_1}{6} \\
		0 & 1 & 1-\gamma_1& \frac{1}{2}-\beta_1& \frac{1}{6}-\frac{\beta_1}{2}& \frac{1}{24}-\frac{\beta_1}{6} \\
		-\tau^2 \lambda  & -\tau^2 \lambda & \alpha_1-1-\frac{1}{2}\tau^2 \lambda&  \alpha_1-1-\frac{1}{6}\tau^2 \lambda&\frac{1}{2}( \alpha_1-1)&\frac{1}{6}(\alpha_1-1) \\
		0 & 0 & 0&1&1&\frac{1}{2}-\beta_2 \\
		0 & 0 & 0&0&1&1-\gamma_2 \\
		0 & 0 & 0&-\tau^2 \lambda&0&\alpha_2-1 
		\end{bmatrix}
		\end{aligned}
		\end{equation}
		and $\bold{F_{n+\alpha_f}}$ consists of the forcing terms $F_{n+\alpha_{f1}}$ and $(F^{(3)}_{n+\alpha_{f2}})$.  Thus, the amplification matrix $G$ becomes
		\begin{equation} \label{eq:ampm}
		G=A^{-1}B,
		\end{equation}
		with $G$ being an upper-block triangular matrix written as:
		\begin{equation}\label{eq:ampblock}
		G=	\begin{bmatrix}
		\Lambda_1 & \Xi\\
		\boldmath{0}&\Lambda_2
		\end{bmatrix},
		\end{equation}
		where
		\begin{equation}\label{eq:block}
		\begin{aligned}
		\Lambda_1&=\dfrac{1}{\alpha_1}\begin{bmatrix}
		\alpha_1-{\beta_1 \tau^2\lambda}& \alpha_1-{\beta_1 \tau^2\lambda}& \frac{1}{2}\left({\alpha_1-\beta_1 \left(\tau^2\lambda+2\right)} \right) \\        
		-{\gamma_1 \tau^2\lambda}& \alpha_1-{ \gamma_1\tau^2\lambda}& \alpha_1-\frac{\gamma_1}{2} \left(\tau^2\lambda+2\right)  \\  
		-{\tau^2\lambda}& -{\tau^2\lambda}& \alpha_1-1 -\frac{\tau^2}{2}\lambda\\
		\end{bmatrix},\\
		\Lambda_2&=\dfrac{1}{\alpha_2}\begin{bmatrix}
		\alpha_2-{\beta_2 \tau^2\lambda}& \alpha_2 & \frac{\alpha_2}{2}-{\beta_2} \\
		-{\gamma_2 \tau^2\lambda} & \alpha_2& \alpha_2-{\gamma_2}\\
		-{\tau^2\lambda} & 0 & {\alpha_2-1}
		\end{bmatrix}.
		\end{aligned}
		\end{equation}
		Following~\cite{ behnoudfar2019higher, behnoudfar2021split}, firstly, we obtain the high-order unknowns $\mathcal{L}^{1}(A_{n})$ and $\mathcal{L}^{2}(A_{n})$, being second-order accurate. Secondly, we set the parameters such that the upper block on the diagonal also delivers second-order accuracy. Then, we obtain the fourth-order accuracy by adding the high-order terms through the upper off-diagonal block to our solution. Therefore, we study each diagonal block to derive the related parameters. For this aim, for any arbitrary amplification matrix, we can state the following,
		\begin{equation} \label{eq:a40}
		G_0 \mathcal{L}^{1}(A_{n+1}) - G_1 \mathcal{L}^{1}(A_n) + G_2 \mathcal{L}^{1}(A_{n-1}) - G_3 \mathcal{L}^{1}(A_{n-2}) = 0,
		\end{equation}
		where the coefficients are invariants of the amplification matrix as $G_0 = 1$, $G_1 $ is the trace of $G$, $G_2$ is the sum of principal minors of $G$, and $G_3$ is the determinant of $G$. Using a Taylor series expansion, we have:
		\begin{equation} \label{eq:te3}
		\begin{aligned}
		\mathcal{L}^{1}(A_{n+1}) & = \mathcal{L}^{1}(A_{n})+\tau \mathcal{L}^{2}(A_{n}) + \frac{\tau^2}{2} \mathcal{L}^{3}(A_{n})+ \mathcal{O}(\tau^3), \\
		\mathcal{L}^{1}(A_{n-1}) & = \mathcal{L}^{1}(A_{n})-\tau \mathcal{L}^{2}(A_{n}) + \frac{\tau^2}{2} \mathcal{L}^{3}(A_{n})+ \mathcal{O}(\tau^3), \\
		\mathcal{L}^{1}(A_{n-2}) & = \mathcal{L}^{1}(A_{n}) -2\tau \mathcal{L}^{2}(A_{n}) + {2\tau^2} \mathcal{L}^{3}(A_{n})+ \mathcal{O}(\tau^3).
		\end{aligned}
		\end{equation}
		Setting $\gamma_2=\frac{1}{2}-\alpha_{f2}+\alpha_2$, we obtain that $\mathcal{L}^{1}(A_{n+1}), \mathcal{L}^{2}(A_{n+1})$ are second-order accurate in time. Next, we rewrite~\eqref{eq:te3} for the unknown~$U_{n+1}$ as  
		\begin{equation} \label{eq:te6}
		\begin{aligned}
		U_{n+1} & = U_n + \tau V_n + \frac{\tau^2}{2} A_n+\mathcal{R} , \\
		U_{n-1} & = U_n - \tau V_n + \frac{\tau^2}{2} A_n+\mathcal{R} , \\
		U_{n-2} & = U_n - 2\tau V_n + 4\frac{\tau^2}{2} A_n +\mathcal{R},
		\end{aligned}
		\end{equation}
		where $\mathcal{R}$ is a function of $\mathcal{L}^{1}(A_{n}), \mathcal{L}^{2}(A_{n})$ and  $\mathcal{L}^{3}(A_{n})$. We neglect this term in the analysis as it is a residual. Thus, we can prove that the remaining terms are second-order accurate in time. Then, we add the residuals to the second-order accurate solution, to produce a truncation error of $\mathcal{O}(\tau^5)$ and consequently, we obtain a fourth-order accurate scheme in time, which completes the proof.
	\end{proof}
	
	\subsection{Stability analysis and CFL condition}\label{sec:anal}
	
	We bound the spectral radius of the amplification matrix by one to deliver a stable time-marching scheme. For this, firstly, we calculate the eigenvalues of the matrix $G$ in~\eqref{eq:ampm} for the case $\Theta:=\tau^2\lambda \to 0$. Therefore, in this case, the diagonal blocks in~\eqref{eq:block} are:
	\begin{equation}\label{eq:zero}
	\begin{aligned}
	\Lambda_1&=\begin{bmatrix}
	1& 1& \frac{1}{2}- \frac{\beta_1}{\alpha_1} \\        
	0& 1& 1-\frac{\gamma_1}{\alpha_1} \\  
	0& 0& 1-\frac{1}{\alpha_1}
	\end{bmatrix},\
	\Lambda_2&=\begin{bmatrix}
	1& 1& \frac{1}{2}- \frac{\beta_2}{\alpha_2} \\        
	0& 1& 1-\frac{\gamma_2}{\alpha_2} \\  
	0& 0& 1-\frac{1}{\alpha_2}
	\end{bmatrix}.
	\end{aligned}
	\end{equation}
	Then, $G$'s eigenvalues are:
	\begin{equation}\label{eq:t0}
	\lambda_1=\lambda_2=\lambda_3=\lambda_4=1,\qquad \lambda_5=1-\frac{1}{\alpha_1}, \qquad \lambda_6=1-\frac{1}{\alpha_2}.
	\end{equation}
	The boundedness of $\lambda_5$ and $\lambda_6$ in~\eqref{eq:t0} implies that $ {\alpha_1}\geq \frac{1}{2} $ and $ {\alpha_2} \geq \frac{1}{2} $.
	
	Now, we derive the CFL condition for our method. In the analysis of the implicit generalized-$\alpha$ methods of second- and higher-order accuracy~\cite{ chung1993time, behnoudfar2019higher}, we analyze the discrete system's eigenvalue distribution in the limit $\Theta \to \infty$ and set the method's free parameters such that all eigenvalues are equal to a real constant $\rho_\infty \in [0,1]$; this parameter controls the numerical dissipation. In our explicit method, we need to find the method's conditional stability region. Thus, we exploit the amplification matrix's upper-triangular structure~\eqref{eq:ampblock} to explicitly compute the eigenvalues of $G$, which is equivalent to finding the eigenvalues of each diagonal block in~\eqref{eq:block}. Therefore, the characteristic polynomials for the diagonal blocks are:
	\begin{equation}\label{eq:charact}
	\begin{aligned}
	\Lambda_1:\qquad
	2\alpha_1\lambda_1^3&+{\lambda_1^2 \left(\Theta (2 \beta_1+2 \gamma_1+1)-6 \alpha_1+2\right)}\\&+{\lambda_1 \left(\Theta (-4 \beta_1-2 \gamma_1+1)+6 \alpha_1-4\right)}+{2 \beta_1 \Theta-2 \alpha_1+2}=0,\\
	\Lambda_2:\qquad
	2\alpha_2\lambda_2^3&+{\lambda_2^2 \left(2\Theta  \beta_2-6 \alpha_2+2\right)}+{\lambda_2 \left(\Theta (-4 \beta_2+2 \gamma_2+1)+6 \alpha_2-4\right)}\\&+{\Theta(2 \beta_2+1-2\gamma_2) -2 \alpha_2+2}=0,
	\end{aligned}
	\end{equation}
	where two roots of each characteristic polynomial in~\eqref{eq:charact} are the principal roots while the third is spurious (unphysical). We require the two principal roots to be complex conjugates except in the high-frequency regions. This requirement maximizes high-frequency dissipation while setting two eigenvalues to one in the low-frequency range to improve its approximation accuracy. Thus, we set $\beta_1$ and $\beta_2$ such that the complex parts of the principal roots of the blocks $\Lambda_1$ and $\Lambda_2$, respectively, vanish in high-frequency regions; this setting changes the largest eigenvalue from one to a user-defined value, which results in an approach similar to the implicit generalized-$\alpha$ methods. Additionally, we define the limit in which this bifurcation happens at block $i,\, i=1, \,2$, by $\Omega_{bi}$. We define the critical stability limit for our explicit method as $\Omega_{si}$ and show the stability region of block~$i$. 
	
	We find the parameter values using the characteristic equation corresponding to the diagonal block~$i$ of the amplification matrix as:
	\begin{equation}\label{eq:char}
	\sum^{3}_{j=0}\left(\tilde{a}_i^j+\tilde{y}_i^j \Theta^2\right)\lambda_i^{(3-j)}=0,
	\end{equation}
	where $\tilde{a}_i$ and $\tilde{y}_i$ are functions of the parameters $\gamma_i$ and $\beta_i$. Each block~$i$ has three eigenvalues; we set two of them to $\rho_{bi}$ and one becomes $\rho_{si}$.  Therefore, we rewrite the characteristic polynomial~\eqref{eq:char} as
	\begin{equation}\label{eq:root}
	\left(\lambda_i^2+\rho_{bi}^2+2\lambda_i\rho_{bi}\right)\left(\lambda_i+\rho_{si}\right)=0.
	\end{equation}
	Next, to have all three roots with real values at the bifurcation limit, we equate~\eqref{eq:root} to~\eqref{eq:charact} to obtain:
	\begin{equation}
	\begin{aligned}
	\Omega_{b1}&=2 - 2\rho_{b1} - \rho_{s1} +\rho_{s1} \rho_{b1}^2,\\
	\Omega_{b2}&=2 + 2\rho_{b2} + \rho_{s2} -\rho_{s2} \rho_{b2}^2,\\
	\alpha_1&=\frac{2 - (1 + \rho_{b1}) \rho_{s1}}{(-1 + \rho_{b1}) (-1 + \rho_{s1})}\\
	\alpha_2&=\frac{2 + (1 - \rho_{b2}) \rho_{s2}}{(1 + \rho_{b2}) (1 + \rho_{s2})}.
	\end{aligned}
	\end{equation}
	\begin{remark}
		Herein, we constrain $\rho_{si} \leq \rho_{bi}$ which results in $\Omega_{bi} \leq \Omega_{si}$. We can maximize the bifurcation region $\Omega_{bi}$ by setting $\rho_{si} = \rho_{bi}$ and consequently, obtain a one-parameter family of algorithms, Figure~\ref{fig:rhovar} shows this numerically.   
	\end{remark}
	
	Figure~\ref{fig:rhovar} shows how the bifurcation and stability regions change as the user-defined parameter changes. 
	
	\begin{figure}[!ht]    
		\begin{subfigure}{.5\textwidth}
			\centering\includegraphics[width=6.6cm]{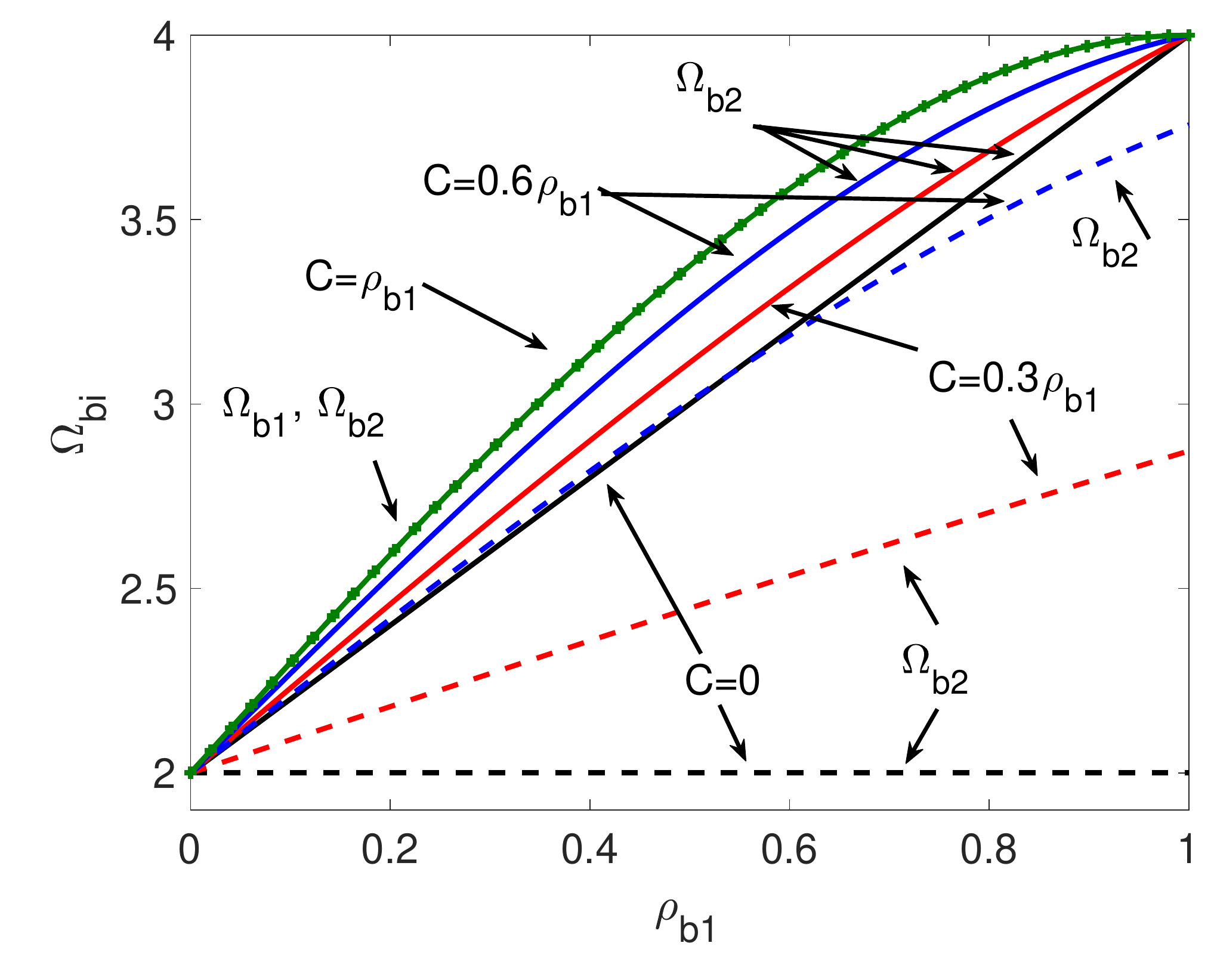}     
			\caption{$\rho_{s1}=\rho_{s2}=\rho_{b2}=\,C$.}
		\end{subfigure}
		\begin{subfigure}{.5\textwidth}
			\centering\includegraphics[width=6.6cm]{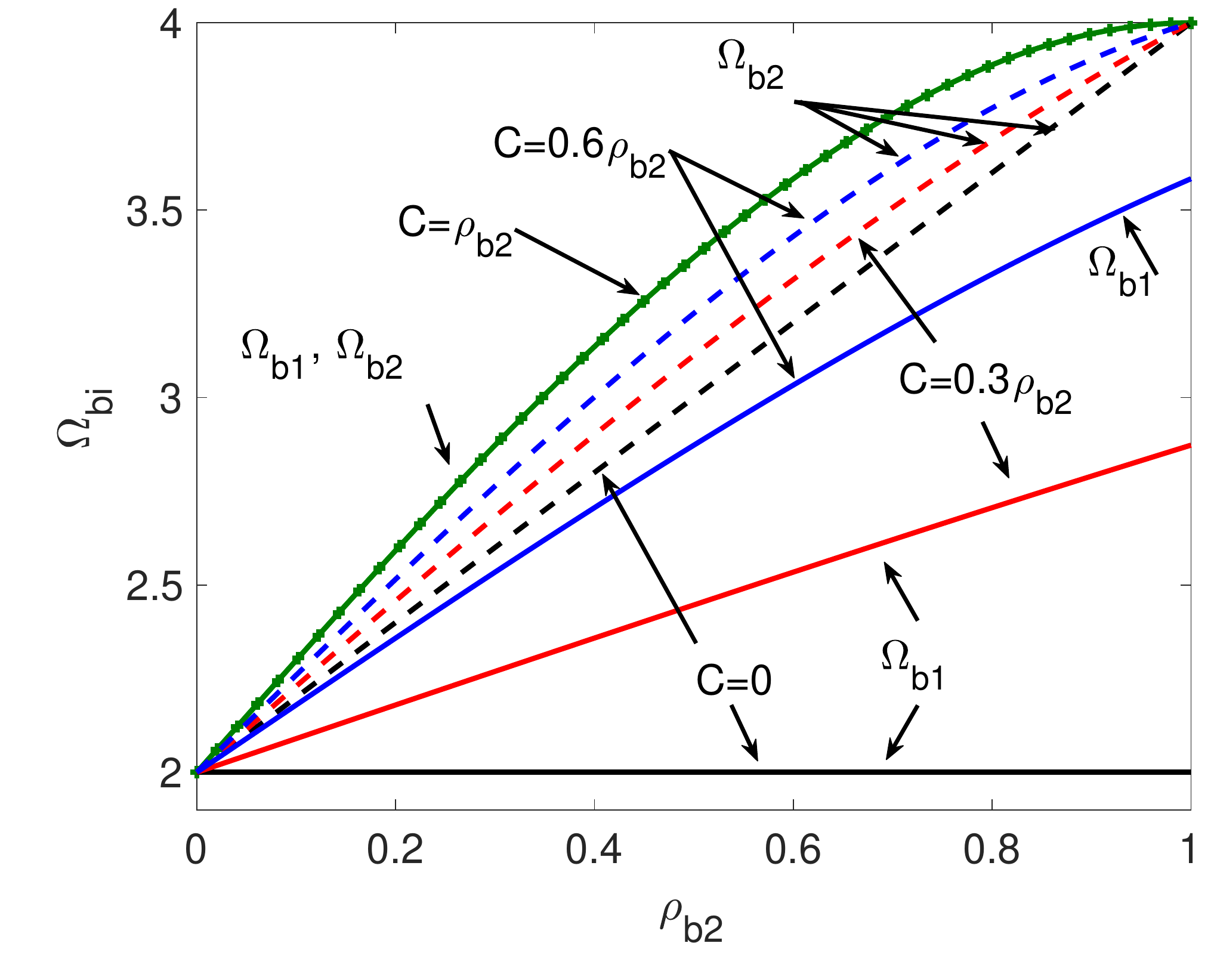} 
			\caption{$\rho_{s2}=\rho_{s1}=\rho_{b1}=\,C$.}
		\end{subfigure}
		\caption{Effect of user-defined parameters on the bifurcation regions $\Omega_{b1}$ and $\Omega_{b2}$.}
		\label{fig:rhovar}
	\end{figure}
	
	To set the complex part of the eigenvalues equal to zero in the high-frequency regions, we define the parameters $\beta_i$ as
	\begin{equation}
	\begin{aligned}
	\beta_1&=\frac{(1 + \rho_{b1})(-1 + \rho_{b1} \rho_{s1})^2}{(-1 + \rho_{b1})^2 (-1 + \rho_{s1})(-2 + \rho_{s1} + \rho_{b1} \rho_{s1})},\\
	\beta_2&=\frac{-5 - 3\rho_{b2} - 4\rho_{s2} + 2\rho_{b2}\rho_{s2} + 2 \rho_{b2}^2\rho_{s2} - \rho_{s2}^2+\rho_{b2}\rho_{s2}^2}{(1 + \rho_{b2})^2(-2 - 3\rho_{s2}+\rho_{b2} \rho_{s2}- \rho_{s2}^2 + \rho_{b2} \rho_{s2}^2)}.
	\end{aligned}
	\end{equation} 
	Finally, we find the critical values at which each block's spectral radius becomes larger than one. For this, we set $\Theta=\Omega_{si}$ and $\lambda_i=1$ in~\eqref{eq:char} and solve the resulting equation. Thus, we introduce the critical values $\Omega_{si}$ as 
	\begin{equation}
	\begin{aligned}
	\Omega_{s1}&=\frac{4(1-\rho_{b1})(2-\rho_{b1}\rho_{s1}-\rho_{s1})(3+\rho_{b1}-\rho_{s1}-3\rho_{b1}\rho_{s1})}{2(5-\rho_{b1}^2)+
		(5-13\rho_{b1}-\rho_{b1}^2-\rho_{b1}^3)\rho_{s1}-(1+\rho_{b1})^3\rho_{s1}^2},\\
	\Omega_{s2}&=\frac{4(1+\rho_{b2})(2-\rho_{b2}\rho_{s2}+\rho_{s2})(3-\rho_{b2}+\rho_{s2}-3\rho_{b2}\rho_{s2})}{2(5-\rho_{b2}^2)+
		(5-13\rho_{b2}-\rho_{b2}^2+\rho_{b2}^3)\rho_{s2}-(1-\rho_{b2})^3\rho_{s2}^2}.
	\end{aligned}
	\end{equation}
	
	\begin{remark}
		Following closely~\cite{ behnoudfar2018variationally, HULBERT1996175}, $\alpha_{f1}$  and $\alpha_{f2}$ are free parameters with respect to spectral radius. Therefore, we set $\alpha_{f1}=1$ to deliver fourth-order accuracy. Accordingly, to optimally combine low- and high-frequency dissipation, we set $\alpha_{f2}=0$.
	\end{remark}
	
	Therefore, by setting $0\leq \rho_{b1},\rho_{b2}\leq1 $, we control the system's spectral radius $\rho$ and, consequently, the high-frequency numerical damping. Figure~\ref{fig:conn} shows how the user controls $\rho$ and the stability region $\Omega_s$; setting $\rho_b=\rho_s=0.99$ leads to the largest stability region $\Omega_s=4(\theta \tau)^2=4\Theta$, equivalent to the stability region of the second-order central difference method~\cite{ hairer2010solving}. 
	
	\begin{figure}[!ht]    
		\includegraphics[width=14cm]{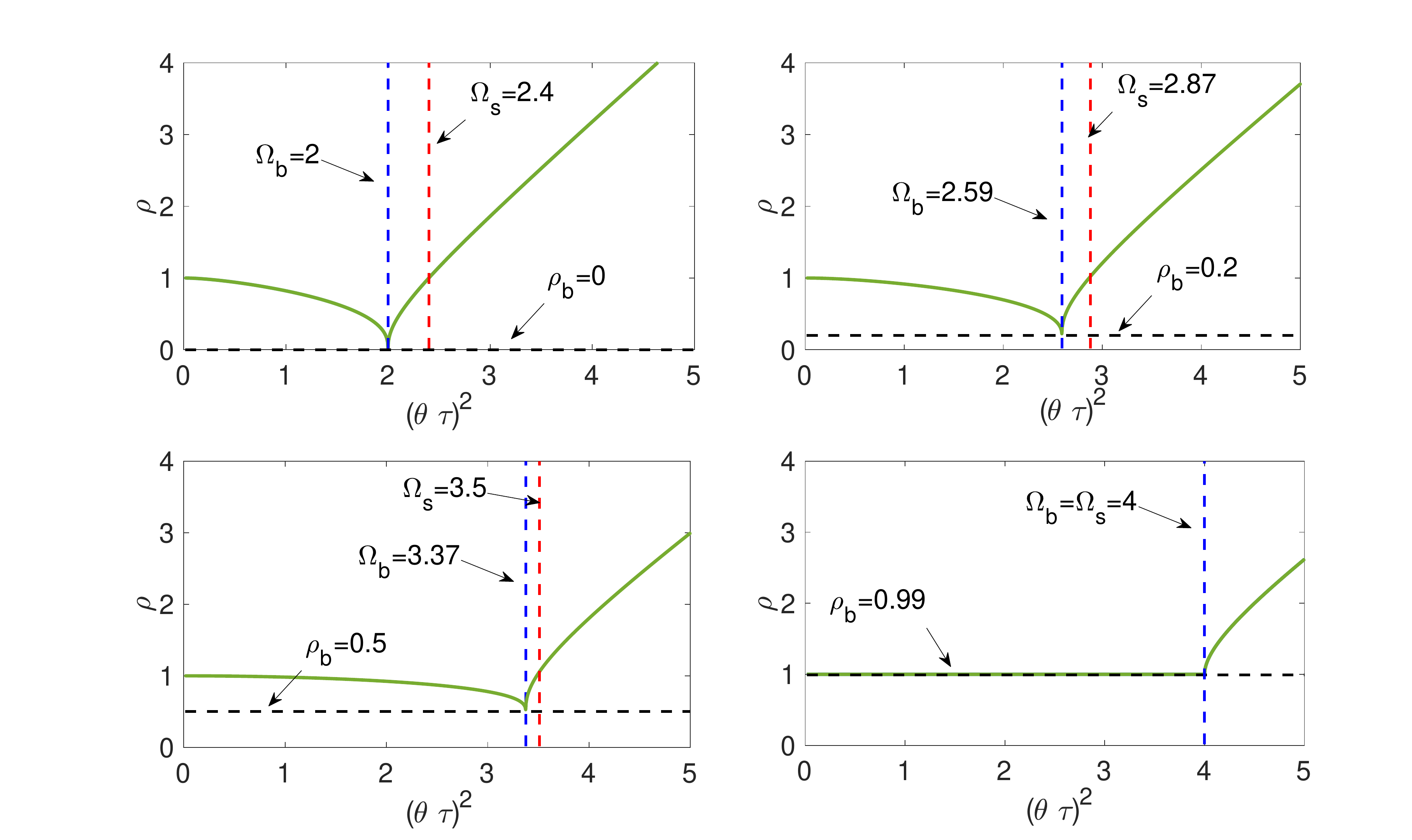}     
		\caption{Spectral radius $\rho$  behaviour for different $\rho_{b}$ values, where $\rho_{b1}=\rho_{b2}=\rho_{s1}=\rho_{s2}$.}
		\label{fig:conn}
	\end{figure}

	\section{General $2k^{th}$-order accuracy in time}\label{sec:4}
	
	This section extends our approach to deliver $2k^{th}$-order accurate methods with $k\ge2$. That is, we introduce our $2k^{th}$-order explicit generalized-$\alpha$ method as:  
	\begin{equation} \label{eq:ho1}
	\begin{aligned}
	MA_{n}^{\alpha_1}+CV_n+KU_{n}&=F_{n+\alpha_{f1}}, \\
	M\mathcal{L}^{3j-3}(A_n)^{\alpha_j}+C\mathcal{L}^{3j-4}(A_n)+K\mathcal{L}^{3j-5}(A_n)&=F^{(3j-3)}_{n+\alpha_{fj}} , ,\qquad j=2, \cdots, k-1, \\
	M\mathcal{L}^{3k-3}(A_n)^{\alpha_k}+C\mathcal{L}^{3k-4}(A_n)+K\mathcal{L}^{3k-5}(A_n)&=F^{(3k-3)}_{n+\alpha_{fk}} ,
	\end{aligned}
	\end{equation} 
	and updating the system using the following
	\begin{equation} \label{eq:ho2}
	\begin{aligned}
	U_{n+1} & = U_n + \tau V_n + \frac{\tau^2}{2} A_n + \frac{\tau^3}{6} \mathcal{L}^{1}(A_n)+ \cdots+\frac{\tau^{3k-1}}{(3k-1) !} \mathcal{L}^{3k-3}(A_n)+\beta_1 \tau^2 P_{n,1}, \\
	V_{n+1}  &= V_n + \tau A_n + \frac{\tau^2}{2}  \mathcal{L}^{1}(A_n) + \cdots+\frac{\tau^{3k-2}}{(3k-2) !} \mathcal{L}^{3k-3}(A_n)+ \tau \gamma_1 P_{n,1}, \\
	\mathcal{L}^{3j-5}(A_{n+1})&= \mathcal{L}^{3j-5}(A_n) + \tau \mathcal{L}^{3j-4}(A_n) + \cdots+ \frac{\tau^{3(k-j)+2}}{(3(k-j)+2) !} \mathcal{L}^{3k-3}(A_n)+ \tau^2 \beta_j P_{n,j}, \\
	\mathcal{L}^{3j-4}(A_{n+1})&= \mathcal{L}^{3j-4}(A_{n}) +  \cdots+ \frac{\tau^{3(k-j)+1}}{(3(k-j)+1) !} \mathcal{L}^{3k-3}(A_{n})+ \tau \gamma_j P_{n,j},\qquad j=2, \cdots, k-1, \\[0.2cm] 
	\mathcal{L}^{3k-5}(A_{n+1})&= \mathcal{L}^{3k-5}(A_n) + \tau {L}^{3k-4}(A_n) + \frac{\tau^2}{2}\mathcal{L}^{3k-3}(A_n) + \tau^2 \beta_k  \llbracket \mathcal{L}^{3k-3}(A_{n} \rrbracket, \\
	\mathcal{L}^{3k-4}(A_{n+1}) &=\mathcal{L}^{3k-4}(A_{n}) + \tau \mathcal{L}^{3k-3}(A_{n}) + \tau \gamma_k \llbracket \mathcal{L}^{3k-3}(A_{n} \rrbracket,
	\end{aligned}
	\end{equation}
	with
	\begin{equation} \label{eq:ho3}
	\begin{aligned}
	P_{n,1}&=A_{n+1}-A_n-\tau\mathcal{L}^{1}(A_{n})-\cdots-\frac{\tau^{3k-3}}{(3k-3) !} \mathcal{L}^{3k-3}(A_{n}),\\
	A_{n}^{\alpha_1}&=A_n+\tau \mathcal{L}^{1}(A_{n})+\cdots+\frac{\tau^{3k-3}}{(3k-3) !} \mathcal{L}^{3k-3}(A_{n})+\alpha_1 P_{n,1},\\
	% &P_{n,j}=\mathcal{L}^{3(j-2)+3}(A_{n+1})-\mathcal{L}^{3(j-2)+3}(A_n)-\tau \mathcal{L}^{3(j-2)+4}(A_n)-\frac{\tau^2}{2} \mathcal{L}^{3(j-2)+5}(A_n)\\&-\frac{\tau^3}{6} \mathcal{L}^{3(j-2)+6}(A_n),\\
	% &A_{n}^{\alpha_j}=\mathcal{L}^{3(j-2)+3}(A_n)+\tau \mathcal{L}^{3(j-2)+4}(A_n)+\frac{\tau^2}{2} \mathcal{L}^{3(j-2)+5}(A_n)\\&+\frac{\tau^3}{6} \mathcal{L}^{3(j-2)+6}(A_n)+\alpha_jP_{n,j},\\
	P_{n,j}&=\mathcal{L}^{3j-3}(A_{n+1})-\mathcal{L}^{3j-3}(A_n)- \cdots- \frac{\tau^{3(k-j)}}{3(k-j) !}\mathcal{L}^{3k-3}(A_{n}),\\
	A_{n}^{\alpha_j}&=\mathcal{L}^{3j-3}(A_n)+\cdots+\frac{\tau^{3(k-j)}}{3(k-j) !}\mathcal{L}^{3k-3}(A_{n})+\alpha_jP_{n,j}, \qquad j=2, \cdots, k-1,\\
	\mathcal{L}^{3k-3}(A_n)^{\alpha_k}&=\mathcal{L}^{3k-3}(A_n)+\alpha_k \llbracket \mathcal{L}^{3k-3}(A_{n}) \rrbracket.
	\end{aligned}
	\end{equation}
	
	Letting $k=2$ recovers the fourth-order explicit generalized-$\alpha$ method of Section~\ref{sec:3}. Using similar arguments to the proof of Theorem~\ref{thm:3o}, we can establish higher-order schemes in the form of~\eqref{eq:ho1} and~\eqref{eq:ho2} by using high-order Taylor series. Then, to seek $2k^{th}$-order of accuracy, we substitute~\eqref{eq:ho2} into~\eqref{eq:ho1} and find a matrix system as
	\begin{equation}\label{eq:eig2}
	L \bfs{U}_{n+1} = R \bfs{U}_n+\bold{F}_{n+\alpha_f}.
	\end{equation}
	Therefore, the amplification matrix of the $2k^{th}$-order accurate scheme becomes:
	\begin{equation}\label{eq:HOampblock}
	G_k=L^{-1}R=\begin{bmatrix}
	\bold{ \Lambda}_1 &\bold{\Xi}_{12} &\cdots&\cdots& \bold{\Xi}_{1k}\\
	\boldsymbol{0}&\bold{\Lambda}_2&\bold{\Xi}_{23} &\cdots& \bold{\Xi}_{2k}\\
	\vdots& &\ddots& \\
	\boldsymbol{0}&\boldsymbol{0}&\cdots & \bold{\Lambda}_{k-1}& \bold{\Xi}_{k-1}\\
	\boldsymbol{0}&\boldsymbol{0}&\cdots & \boldsymbol{0}&\bold{ \Lambda}_k
	\end{bmatrix},
	\end{equation}
	with
	\begin{equation}\label{eq:HOblock}
	\begin{aligned}
	\bold{\Lambda}_j&=\dfrac{1}{\alpha_j}\begin{bmatrix}
	\alpha_j-{\beta_j \Theta}& \alpha_j-{\beta_j \Theta}& \frac{1}{2}\left({\alpha_j-\beta_j \left(\Theta+2\right)} \right) \\        
	-{\gamma_j \Theta}& \alpha_j-{\gamma_j \Theta}& \alpha_j-\frac{\gamma_j}{2} \left(\Theta+2\right)  \\  
	-{\Theta}& -{\Theta}& \alpha_j-1 -\frac{\Theta}{2}\\
	\end{bmatrix},\qquad j=1, \cdots, k-1\\
	\bold{\Lambda}_k&=\dfrac{1}{\alpha_k}\begin{bmatrix}
	\alpha_k-{\beta_k \Theta}& \alpha_k & \frac{\alpha_k}{2}-{\beta_k} \\
	-{\gamma_k \Theta} & \alpha_k& \alpha_k-{\gamma_k}\\
	-{\Theta} & 0 & {\alpha_k-1}
	\end{bmatrix}.
	\end{aligned}
	\end{equation}
	
	\begin{remark}
		The amplification matrix~\eqref{eq:HOampblock} is an upper-triangular block matrix; thus, we neglect the non-diagonal block contributions in the eigenvalue analysis.
	\end{remark}
	
	\begin{thm}
		Assuming $u_h(t)$ and $\dot{u}_h(t)$ have sufficient regularity in time, our semi-discrete method~\eqref{eq:ho1}-\eqref{eq:ho3} for advancing~\eqref{eq:mp} is  $2k^{th}$-order accurate in time when
		\begin{equation} \label{eq:HOgamma}
		\begin{aligned}
		\gamma_j&=\frac{1}{2}-\alpha_{fj}+\alpha_j,\qquad j=1, \cdots, k-1,\quad
		\gamma_k&=\frac{1}{2}-\alpha_{fk}+\alpha_k.
		\end{aligned}
		\end{equation}
		\begin{proof}
			The amplification matrix~\eqref{eq:HOampblock} is an upper-diagonal block matrix; each block is a $3 \times 3$ matrix. The last two diagonal blocks $\Lambda_{k-1}$ and $\Lambda_{k}$ have similar entries to those of the amplification matrix of the fourth-order method. The other diagonal blocks are analogous to $\Lambda_{k-1}$. Hence, separately for each block, we determine the relevant terms for~\eqref{eq:a40} and consider the higher-order terms to obtain second-order accuracy. Consequently, after solving the whole system and adding the higher-order terms to the unknowns $u_h$ and $v_h$, we have a truncation error of $\mathcal{O}(\tau^{2k+1})$.
		\end{proof}
	\end{thm}
	
	\subsection{CFL condition and dissipation control}
	
	We define the system's spectral behaviour, the bifurcation regions, and its CFL conditions; thus, we study the amplification matrix's eigenvalues~\eqref{eq:HOampblock}. Accordingly, we calculate each diagonal block's eigenvalues. The first $k-1$ blocks have identical structures; hence, following Section~\ref{sec:3}, we propose these algorithmic parameters as
	\begin{equation}
	\boxed{
		\begin{aligned}
		\Omega_{bj}&=2 - 2\rho_{bj} - \rho_{sj} +\rho_{sj} \rho_{bj}^2, \qquad j=1,\cdots,k-1,\\
		\alpha_j&=\frac{2 - (1 + \rho_{bj}) \rho_{sj}}{(-1 + \rho_{bj}) (-1 + \rho_{sj})}, \\
		\beta_j&=\frac{(1 + \rho_{bj})(-1 + \rho_{bj} \rho_{sj})^2}{(-1 + \rho_{bj})^2 (-1 + \rho_{sj})(-2 + \rho_{sj} + \rho_{bj} \rho_{sj})},\\
		\Omega_{bk}&=2 + 2\rho_{bk} + \rho_{sk} -\rho_{sk} \rho_{bk}^2,\\
		\alpha_k&=\frac{2 + (1 - \rho_{bk}) \rho_{sk}}{(1 + \rho_{bk}) (1 + \rho_{sk})},\\
		\beta_k&=\frac{-5 - 3\rho_{bk} - 4\rho_{sk} + 2\rho_{bk}\rho_{sk} + 2 \rho_{bk}^2\rho_{sk} - \rho_{sk}^2+\rho_{bk}\rho_{sk}^2}{(1 + \rho_{bk})^2(-2 - 3\rho_{sk}+\rho_{bk} \rho_{sk}- \rho_{sk}^2 + \rho_{bk} \rho_{sk}^2)},
		\end{aligned}}
	\end{equation}
	and the critical values $\Omega_{sj}$ are 
	\begin{equation}
	\begin{aligned}
	\Omega_{sj}&=\frac{4(1-\rho_{bj})(2-\rho_{bj}\rho_{sj}-\rho_{sj})(3+\rho_{bj}-\rho_{sj}-3\rho_{bj}\rho_{sj})}{2(5-\rho_{bj}^2)+
		(5-13\rho_{bj}-\rho_{bj}^2-\rho_{bj}^3)\rho_{sj}-(1+\rho_{bj})^3\rho_{sj}^2},\quad j=1,\cdots,k-1,\\
	\Omega_{sk}&=\frac{4(1+\rho_{bk})(2-\rho_{bk}\rho_{sk}+\rho_{sk})(3-\rho_{bk}+\rho_{sk}-3\rho_{bk}\rho_{sk})}{2(5-\rho_{bk}^2)+
		(5-13\rho_{bk}-\rho_{bk}^2+\rho_{bk}^3)\rho_{sk}-(1-\rho_{bk})^3\rho_{sk}^2}.\\\\
	\end{aligned}
	\end{equation}
	
	\begin{remark}
		We constrain $\rho_{sm} \leq \rho_{bm}, \, m=1, \cdots, k$ and maximize the bifurcation regions $\Omega_{bm}$ by setting $\rho_{sm} = \rho_{bm}=\rho $ with $0 \leq \rho <1$ as a user-defined parameter; our method is a one-parameter family of time-marching algorithms. Additionally, the stability region is independent of the accuracy order; we obtain higher-order accuracy  without affecting any features of the second-order algorithm (i.e., preserve stability regions, bifurcation limit, and dissipation control).  Similarly, $\alpha_{fj}$ and $\alpha_{fk}$ are free parameters; thus, we set $\alpha_{fj}=1, \,\,j=1,\,\cdots,\,k-1$, and $\alpha_{fk}=0$.
	\end{remark}
	
	\begin{remark}
		The most expensive computational cost of our explicit time marching is the factorization of the mass matrices. Next, we discuss a state-of-the-art approach to precondition the mass matrices resulting from isogeometric analysis that minimize this cost for complex geometries as well as for single- and multi-patch discretizations.
	\end{remark}

	\section{Solver}\label{sec:5}
	
	The method we propose has many advantages: for instance, the CFL condition and the dissipation control are independent of the order of accuracy; however, it requires solving $k$   mass-matrix systems.  Herein, we discuss a method to accelerate these systems' solution when using isogeometric analysis. Firstly, we describe efficient and robust preconditioners for maximum-continuity isogeometric mass matrices for single and multi-patch geometries. Then, we adopt an iterative solver (i.e., preconditioned conjugate gradients, PCG) to calculate the solutions.
	
	\subsection{B-splines}
	
	Given two positive integers $p$ and $m$, consider an open knot vector $$\Xi := \{\xi_1,\ldots, \xi_{m+p+1}\}$$ such that
	$$
	\xi_1 =\ldots=\xi_{p+1} < \xi_{p+2} \le \ldots \le 
	\xi_{m} < \xi_{m+1}=\ldots=\xi_{m+p+1},
	$$
	where interior repeated knots are allowed with maximum multiplicity $p$. We assume $\xi_1 = 0$ and $\xi_{m+p+1} =1 $. From the knot vector $\Xi$, we define degree-$p$  B-spline functions  using the Cox-De Boor recursive formula: we start with piecewise constants ($p=0$):
	\begin{equation*}%\label{eq:cox-deboor-1}
	\widehat b_{i,0}(\zeta) = \left \{
	\begin{array}{ll}
	1 & \text{if } \xi_i \leq \zeta < \xi_{i+1}, \\
	0 & \text{otherwise},
	\end{array}
	\right.
	\end{equation*}
	for $p \ge 1$, the following recursion defines the B-spline functions
	\begin{equation*}%\label{eq:cox-deboor-2}
	\widehat b_{i,p}(\zeta) = \frac{\zeta - \xi_i}{\xi_{i+p} - \xi_i}  \widehat b_{i,p-1}(\zeta) + \frac{\xi_{i+p+1} - \zeta}{\xi_{i+p+1} - \xi_{i+1}}  \widehat b_{i+1,p-1}(\zeta),
	\end{equation*}
	where $0/0 = 0$. Each B-spline $ \widehat b_{i,p}$ depends only on $p+2$ knots, which we collect in a local knot vector
	\begin{equation*}
	\Xi_{i,p}:= \{ \xi_{i}, \ldots,\, \xi_{i+p+1}\},
	\end{equation*}
	is non-negative and its support is the interval $[\xi_i , \xi_{i+p+1}]$. Moreover, these \mbox{B-spline} functions constitute a partition of unity, that is
	\begin{align} \label{eq:partition_unity}
	\sum_{i=1}^{m} \widehat b_{i,p} (x)=1, & & \forall x \in (0,1).
	\end{align}
	The univariate spline space is
	\begin{equation*}
	\widehat{\mathcal{S}}_{h} = \widehat{\mathcal{S}}_{h}([0,1]) : = \mathrm{span}\{\widehat{b}_{i,p}\}_{i = 1}^m,
	\end{equation*}
	where $h$ denotes the maximal mesh-size. We may drop the degree $p$ from the notation when it will not lead to confusion (see, e.g.,~\cite{ Cottrell2009, DeBoor2001}).
	
	We define multivariate B-splines from univariate ones by tensorization, as is common practice. Let $d$ be the space dimension and consider open knot vectors ${\Xi_k = \{\xi_{k,1}, \ldots, \xi_{k,m + p + 1} \}}$ and a set of multi-indices ${\mathbf{I}:=\{ \mathbf{i}=(i_1,\ldots, i_d): \, 1 \leq i_l \leq m \}}$. For each multi-index $\mathbf{i}=(i_1,\ldots, i_d)$, we introduce the $d$-variate B-spline,
	\begin{equation*}
	\label{eq:multivariate-B-splines}
	\widehat B_{\mathbf{i}}(\mathbf{\zeta}) :=  \widehat
	b[\Xi_{i_1,p}](\zeta_1) \ldots  \widehat
	b[\Xi_{i_d,p}](\zeta_d).
	\end{equation*}
	The  corresponding spline space is
	\begin{equation*}
	\widehat{\mathcal{S}}_{h} =\widehat{\mathcal{S}}_{h}([0,1]^d)  := \mathrm{span}\left\{B_{\mathbf{i}} : \, \mathbf{i} \in \mathbf{I} \right\},
	\end{equation*} 
	where $h$ is the maximal mesh-size in all dimensions, that is,
	\begin{equation*}
	h:= \max_{\substack{1 \leq k \leq d \\ 1 \leq i \leq m+p+1 }}\{ | \xi_{k,i+1} - \xi_{k,i} |\}.
	\end{equation*}
	
	\begin{ass}\label{ass:quasi_uniform_mesh}
		Knot vectors are quasi-uniform; there exists $\alpha > 0$, $h$ independent, such that each non-empty knot span $( \xi_{k,i} , \xi_{k,i+1})$ fulfills $ \alpha h \leq \xi_{k,i+1} - \xi_{k,i}$, for $1\leq k \leq d$.
	\end{ass}
	
	\subsection{Single-patch geometric space}
	
	We consider a single-patch domain $\Omega \subset \mathbb{R}^d$, a $d$-dimensional parametrization $\vett[F]$, 
	\begin{align*}%\label{parametrization}
	\Omega = \boldsymbol{F} (\widehat{\Omega}), & & \text{with } \boldsymbol{F}(\boldsymbol{\xi}) = \sum_{\mathbf{i}}  \boldsymbol{C}_{\mathbf{i}}  \widehat  B_{\mathbf{i}} (\boldsymbol{\xi}),
	\end{align*}
	where $\boldsymbol{C}_{\mathbf{i}} $ are the control points and $ \widehat B_{\mathbf{i}}$ are tensor-product B-spline basis functions defined on a parametric patch $\widehat{\Omega}:=(0,1)^d$, where $\vett[F]$ is an invertible map. Following the isoparametric paradigm, isogeometric basis functions $B_{\mathbf{i}}$ are the push-forward of the parametric basis, that is, $B_{\mathbf{i}} = \widehat B_{\mathbf{i}}\circ \boldsymbol{F} ^{-1}$.  Thus, the isogeometric space on $\Omega$ is defined as
	\begin{equation*}\label{eq:disc_space}
	\mathcal{S}_{h} = \mathcal{S}_{h}(\Omega) := \mathrm{span}\left\{  B_{\mathbf{i}}:=\widehat B_{\mathbf{i}} \circ \mathbf{F}^{-1} \ : \ \mathbf{i} \in \mathbf{I}   \right\}. 
	\end{equation*}
	We introduce a co-lexicographical reordering of the basis functions and write 
	\begin{equation}\label{def:parametric_spline}
	\mathcal{S}_{h}= \spann \left\{  B_{\mathbf{i}} : \, \mathbf{i} \in \mathbf{I} \right\} = \spann\left\{ B_{i} \right\}_{i=1}^{\Ndof}.
	\end{equation}
	
	\subsection{Multi-patch B-splines}\label{sec:mp}
	
	Following~\cite{ daveiga_buffa_sangalli_2014}, a multi-patch domain $\Omega \subset \mathbb{R}^d$ is an open set,  a subdomain union
	\begin{equation}\label{def:patches}
	\overline{\Omega}= \bigcup_{{\ptca}=1}^{\NOmega} \overline{\Omega^{({\ptca})}},
	\end{equation} 
	where $\NOmega$ is the number of subdomains, $\Omega^{({\ptca})}=\vett[F]^{({\ptca})}(\widehat{\Omega})$ are the disjoint patches (subdomain pre-images), each $\vett[F]^{({\ptca})}$ has a different spline parametrization, and the super index $({\ptca})$  refers to $\Omega^{({\ptca})}$. We introduce for each patch $\Omega^{({\ptca})}$, B-spline spaces
	\begin{align*}
	\widehat{\mathcal{S}}^{({\ptca})}_{h}:= \spann\left\{ \widehat{B}^{({\ptca})}_{i} : \, i=1, \ldots , \Ndof^{({\ptca})} \right\}. 
	\end{align*}
	and isogeometric spaces
	\begin{align*}
	\mathcal{S}^{({\ptca})}_{h}:= \spann\left\{  B^{({\ptca})}_{i} \ : \ i=1,\ldots , N_{\text{dof}}^{({\ptca})}   \right\}.
	\end{align*}
	We assume for simplicity that all patches have the same degree $p$. We define an isogeometric space on $\Omega$ by imposing continuity at the interfaces between patches, that is
	\begin{equation}\label{def:global_space}
	V_h:=\left\{ v \in C^0(\Omega) : v|_{\Omega^{({\ptca})}} \in \mathcal{S}^{({\ptca})}_{h} \text{ for }{\ptca}=1, \ldots , \NOmega \right\}.
	\end{equation}
	We assume a suitable conformity to construct a basis for space $V_{h}$; for all $\ptca, \ptcb \in \{ 1,\ldots , \NOmega \}$, with $\ptca \neq \ptcb$, let $\Gamma_{{\ptca} {\ptcb}}=\partial \Omega^{({\ptca})} \cap \partial \Omega^{({\ptcb})}$ be the interface between the patches $\Omega^{({\ptca})}$ and $\Omega^{({\ptcb})}$.
	\begin{ass}\label{ass:conformity}
		We assume:
		\begin{enumerate}
			\item $\Gamma_{{\ptca}{\ptcb}}$ is either a vertex or the image of a full edge or the image of a full face for both parametric domains.
			\item For each $B^{({\ptca})}_{\mathbf{i}} \in \mathcal{S}_{h}^{({\ptca})}$ such that $\supp(B^{({\ptca})}_{\mathbf{i}}) \cap \Gamma_{{\ptca}{\ptcb}} \neq \emptyset$, there exists a function $B^{({\ptcb})}_{\mathbf{j}} \in \mathcal{S}_{h}^{({\ptcb})}$ such that $B^{({\ptca})}_{\mathbf{i}} |_{ \Gamma_{{\ptca}{\ptcb}}}=B^{({\ptcb})}_{\mathbf{j}} |_{\Gamma_{{\ptca}{\ptcb}}}$.
		\end{enumerate}
	\end{ass}
	We define, for each patch $\Omega^{({\ptca})}$, an application
	\begin{equation*}
	G^{({\ptca})}: \{ 1, \ldots , \Ndof^{(\ptca)} \} \rightarrow \mathcal{J}=\{ 1, \ldots , \dim(V_{h}) \},
	\end{equation*}
	such that $G^{({\ptca})}(i)=G^{({\ptcb})}(j)$ if and only if $\Gamma_{{\ptca}{\ptcb}} \neq \emptyset$ and \mbox{$B^{({\ptca})}_{i} |_{\Gamma_{{\ptca}{\ptcb}}}=B^{({\ptcb})}_{j} |_{\Gamma_{{\ptca}{\ptcb}}}$}. Moreover, we define, for each global index $l \in \mathcal{J}$, a set of pairs $\mathcal{J}_l:= \{ ({\ptca},i): \, G^{({\ptca})}(i)=l \}$, which collects local indices of patch-wise contributions to a global function, and the scalar 
	\begin{equation}\label{def:n_l}
	n_l:= \# \mathcal{J}_l,
	\end{equation}
	that expresses the patch multiplicity for the global index $l$. Furthermore, let
	\begin{equation}\label{def:N_patch}
	\Npatch:=\max \{ n_l : \, l \in \mathcal{J} \}
	\end{equation}
	be the maximum number of adjacent patches (i.e., those with non-empty closure intersection). We define, for each $l \in \mathcal{J}$, the global basis function
	\begin{equation}\label{def:mp_basis}
	B_l(\vett[x]):= \begin{cases}
	B^{({\ptca})}_{i}(\vett[x]) & \text{ if } \vett[x] \in \overline{\Omega ^{({\ptca})}} \text{ and } ({\ptca},i) \in \mathcal{J}_l, \\
	0 & \text{ otherwise},
	\end{cases}
	\end{equation} 
	which is continuous due to Assumption~\ref{ass:conformity}. Then
	\begin{equation}\label{def:V_h}
	V_h= \spann \{ B_l: \, l \in \mathcal{J} \}.
	\end{equation}
	The set $\{ B_l: \, l \in \mathcal{J} \}$
	where $B_l$ is defined as in~\eqref{def:mp_basis}, represents a basis for $V_{h}$. Finally, we define the index set $\mathcal{J}^{(\ptca)} \subset \mathcal{J}$ such that $l \in \mathcal{J}^{(\ptca)}$ if and only if $l=G^{(\ptca)}(i)$ for some~$i$, where $\# \mathcal{J}^{(\ptca)}= \Ndof^{(\ptca)}$ and $\mathcal{J}^{(\ptca)}$ are the index set for $\widehat{\mathcal{S}}_h^{(\ptca)}$ and $\mathcal{S}_h^{(\ptca)}$, abusing notation.
	
	\subsection{Mass preconditioner on a patch} \label{sec:single_patch}
	
	In this section, we briefly revisit the preconditioner described in~\cite{ loli2020}, for an isogeometric mass matrix associated with a single-patch domain, denoted $\Omega$, that is
	\begin{equation}\label{eq:mass_matrix}
	[\M]_{i,j}=\int_{\widehat{\Omega}}  \widehat{B}_{i}  \widehat{B}_{j} |\jac{\vett[F]}|.
	\end{equation}
	\begin{ass}\label{ass:g_regularity}
		We assume $\vett[F] \in C^1([0,1]^d)$ and there exists $\delta>0$ such that for all $\vett[x] \in [0,1]^d$, $\jac{\vett[F]^{(\ptca)}} \geq \delta$.
	\end{ass}
	Hence, as a preconditioner for the mass matrix $\M$, we consider
	\begin{equation}\label{def:single_patch_prec}
	\P:=\Prs,
	\end{equation}
	where
	\begin{align}\label{def:matrices_1}
	[\widehat{\M}]_{i,j}:=\int_{\widehat{\Omega}} \widehat{B}_{i} \widehat{B}_{j}, & & \widehat{\D}:=\text{diag}\left( \widehat{\M} \right),  & & \D:=\text{diag}\left( \M \right).
	\end{align}
	We define
	\begin{align}\label{def:cond}
	\kappa(\mathbf{A}):= \frac{\lambda_{\mathrm{max}}(\mathbf{A})}{\lambda_{\mathrm{min}}(\mathbf{A})},
	\end{align}
	for a symmetric, positive definite matrix $\mathbf{A}$. This preconditioner~\eqref{def:single_patch_prec} has several good properties (see~\cite{ loli2020} for details):
	\begin{itemize}
		\item asymptotic exactness, that is
		\begin{equation}
		\label{eq:limit}
		\lim_{h \rightarrow 0}\kappa(\PMP) =1;
		\end{equation}
		\item $p$-robustness, in fact numerical tests show a slow growth
		(almost linear) of \linebreak$\kappa(\PMP)$ with respect to the spline degree $p$.
		\item good behaviour with respect to the spline parametrization $\vett[F]$; numerical evidence shows that $\kappa(\PMP)$ and the number of PCG iteration needed to converge are small even for distorted geometries.
	\end{itemize}
	
	\subsection{Mass preconditioner on multi-patch domain}\label{sec:multi-patch}
	
	We introduce a mass matrix preconditioner for multi-patch domains, that is
	\begin{align*}
	[M]_{i,j}=\int_{ \Omega} B_iB_j, & & B_i,B_j \in \Bmp,
	\end{align*}
	where $\Omega$ is the union of $\Omega^{(\ptca)}$, see~\eqref{def:patches}. Following~\cite{ loli2020}, we combine the single-patch preconditioner of~\eqref{def:single_patch_prec} with an additive Schwarz method. We define a family of local spaces
	\begin{align}\label{def:V_h_r}
	V_h^{(\ptca)}:= \spann \left\{ \Bmp^{(\ptca)} \right\}, & & \ptca = 1, \ldots , \NOmega,
	\end{align}
	where we have set
	\begin{equation*}
	\Bmp^{(\ptca)}:=\{ B_l :\, l \in \mathcal{J}^{(\ptca)}\},
	\end{equation*}
	with $\mathcal{J}^{(\ptca)}$ defined as in Section~\ref{sec:mp}. Therefore, $V_h^{(\ptca)}$ is the subspace of $V_h$ spanned by the B-splines whose support intersect $\Omega^{(\ptca)}$. Moreover, following the notation of~\cite{ toselli2006domain}, we consider restriction operators $R^{(\ptca)}: V_h \rightarrow V_h^{(\ptca)}$ with $\ptca = 1, \ldots , \NOmega$, defined by
	\begin{align*}
	R^{(\ptca)} \left( \sum_{l \in \mathcal{J}} u_l B_l \right)= \sum_{l \in \mathcal{J}^{(\ptca)}} u_l B_l
	\end{align*}
	and their transposes, in the basis representation, ${{R^{(\ptca)}}^T: V_h^{(\ptca)} \rightarrow V_h}$ correspond, in our case, to the inclusion of $V_h^{(\ptca)}$ into $V_h$. We denote with $\mathbf{R}^{(\ptca)}$ and $\mathbf{R}^{(\ptca)^T}$ the rectangular matrices associated to ${R^{(\ptca)}}$ and ${R^{(\ptca)}}^T$, respectively.
	
	The additive Schwarz preconditioner (inverse) is 
	\begin{equation}\label{def:asp}
	\P^{-1}_{\text{ad}}:= \sum_{{\ptca}=1}^{\NOmega} {\mathbf{R}^{(\ptca)}}^T {\P^{(\ptca)}}^{-1}\mathbf{R}^{(\ptca)},
	\end{equation}
	where we set
	\begin{equation}\label{def:loc_forms2}
	\P^{(\ptca)}:=    \Prsa
	\end{equation}
	and 
	\begin{align}\label{def:loc_matrices}
	% \begin{split}
	[\widehat{\M}^{(\ptca)}]_{i,j}:=\int_{\widehat{\Omega}} \widehat{B}^{(\ptca)}_{i} \widehat{B}^{(\ptca)}_{j}, & & \widehat{\D}^{(\ptca)}:=\text{diag}\left( \widehat{\M}^{(\ptca)} \right), & & 	\D^{(\ptca)}:=\text{diag}\left( \M^{(\ptca)} \right),
	%\end{split}
	\end{align}
	assumption the basis functions ordering of $\{\widehat{B}^{(\ptca)}_{i}\}_{i=1}^{{\Ndof^{(\ptca)}}}$ and $\{{B}^{(\ptca)}_{i}\}_{i=1}^{\Ndof^{(\ptca)}}$ follows Section~\ref{sec:mp}. 
	\begin{ass}\label{ass:multi-patch}
		For all $\ptca=1, \ldots, \NOmega$, let us assume $\vett[F]^{(\ptca)}$ fulfils Assumption~\ref{ass:g_regularity}.
	\end{ass}
	An upper bound for the conditioning of the preconditioned system~\cite{ loli2020} is
	\begin{align*}
	k\left( \PadMPad \right) \leq C \Npatch^2,
	\end{align*}
	where the constant $C$ is independent of $h$ and $\NOmega$; equation~\eqref{def:N_patch} defines $\Npatch$. Numerical experiments show the method's $p$-robustness and good behavior with respect to the spline parametrizations $\vett[F]^{(\ptca)}$.
	
	\subsection{Preconditioners application cost}\label{sec:application_cost}
	
	The mass matrices and preconditioners of Section~\ref{sec:single_patch} and~\ref{sec:multi-patch} are symmetric and positive definite; thus, we adopt the preconditioned conjugate gradient method (PCG) to solve the associated linear systems. The application of the single-patch preconditioner requires the solution of a linear system associated with
	\begin{align*}
	\P=\Prs.
	\end{align*}
	the Kronecker product structure of the matrix implies that the preconditioner's application cost is proportional to $\Ndof$~\cite{gao2014fast, los2017application, gao2015preconditioners, loli2020}.
	
	For the multi-patch cases, the application of $\Prad$,~\eqref{def:asp}, involves, for $\ptca \in \{ 1, \ldots , \NOmega \}$, the application of the operators $R^{(\ptca)}$ and ${R^{(\ptca)}}^T$, whose cost is negligible, and the application of $\left({\P^{(\ptca)}}\right)^{-1}$, whose cost we analyze above. In conclusion, also for $\Prad$, the cost of application is $O(\Ndof)$ floating-point operations (FLOPS).
	
	\section{Numerical results}\label{sec:6}	
	
	In this section, firstly, we show the performance of the proposed explicit \mbox{generalized-$\alpha$} method. Then, we provide further numerical results on our solver. All the tests are performed with Matlab R2015a and GeoPDEs toolbox~\cite{ Vazquez2016}.  To show how our explicit method and solver work, the linear systems are solved by PCG, with tolerance equal to $10^{-12}$ and with the null vector as the initial guess. We denote by $n_{\mathrm{sub}}$ the number of subdivisions, which are the same in each parametric direction and in each patch, by $p$ the spline degree and by $\tau$ the size of the time grid. Moreover, we underline that we only consider the splines of maximal regularity. The symbol ``*'' denotes the impossibility of the formation of the matrix $\M$, due to memory requirements.
	
	In all the examples, we consider the following hyperbolic problem
	\begin{align}\label{eq:hyp_problem}
	\begin{cases}
	\ddot{u}(\boldsymbol{x},t) - \nabla \cdot(\omega^2\, \nabla u(\boldsymbol{x}, t)) & = f(\boldsymbol{x},t), \qquad \,\, \,(\boldsymbol{x}, t) \in \Omega \times (0, T], \\
	u(\boldsymbol{x}, t) & = u_D, \qquad \qquad \boldsymbol{x} \in \partial \Omega, \\
	u(\boldsymbol{x}, 0) & = u_0, \qquad \,\qquad \boldsymbol{x} \in \Omega, \\
	\dot{u}(\boldsymbol{x},0) & = v_0, \qquad\,\, \qquad \boldsymbol{x} \in \Omega,
	\end{cases}
	\end{align}
	where $\omega$ is assumed to be one except in the problem taken into account in section~\ref{sec:dis}. 
	
	\subsection{Convergence of the generalized-$\alpha$ method}
	
	For verifying the accuracy of the $4^{th}$ order explicit generalized-$\alpha$ method, we solve~\eqref{eq:hyp_problem} on $\Omega=[0,1]\times [0,1]$, choosing the source term, Dirichlet boundary condition and initial condition such that the exact solution reads
	\begin{align*}
	u(x,y,t)=\sin (10\pi x) \sin (10\pi x)\left[ \cos\left(10\sqrt{2} \pi t\right) + \sin\left(10\sqrt{2} \pi t\right) \right].
	\end{align*}
	In Figure~\ref{fig:l2_err_tau}, we show the $4^{th}$ order of convergence in the time domain for the unknown displacement and velocity at final time $T=0.1$. In this example, we use $100\times 100$ isogeometric elements with polynomial degree of $p=7$ with continuity $C^6$. We also set all the user-defined parameters $\rho_{bi}=\rho_{si}=\rho$. As the analysis showed before, the obtained solutions $u^n_h$ and $v^n_h$ which respectively approximate $u(x,t_n)$ and $\dot{u}(x,t_n)$ converge with an order of four.
	
	\begin{figure}
		\includegraphics[width=0.49\textwidth]{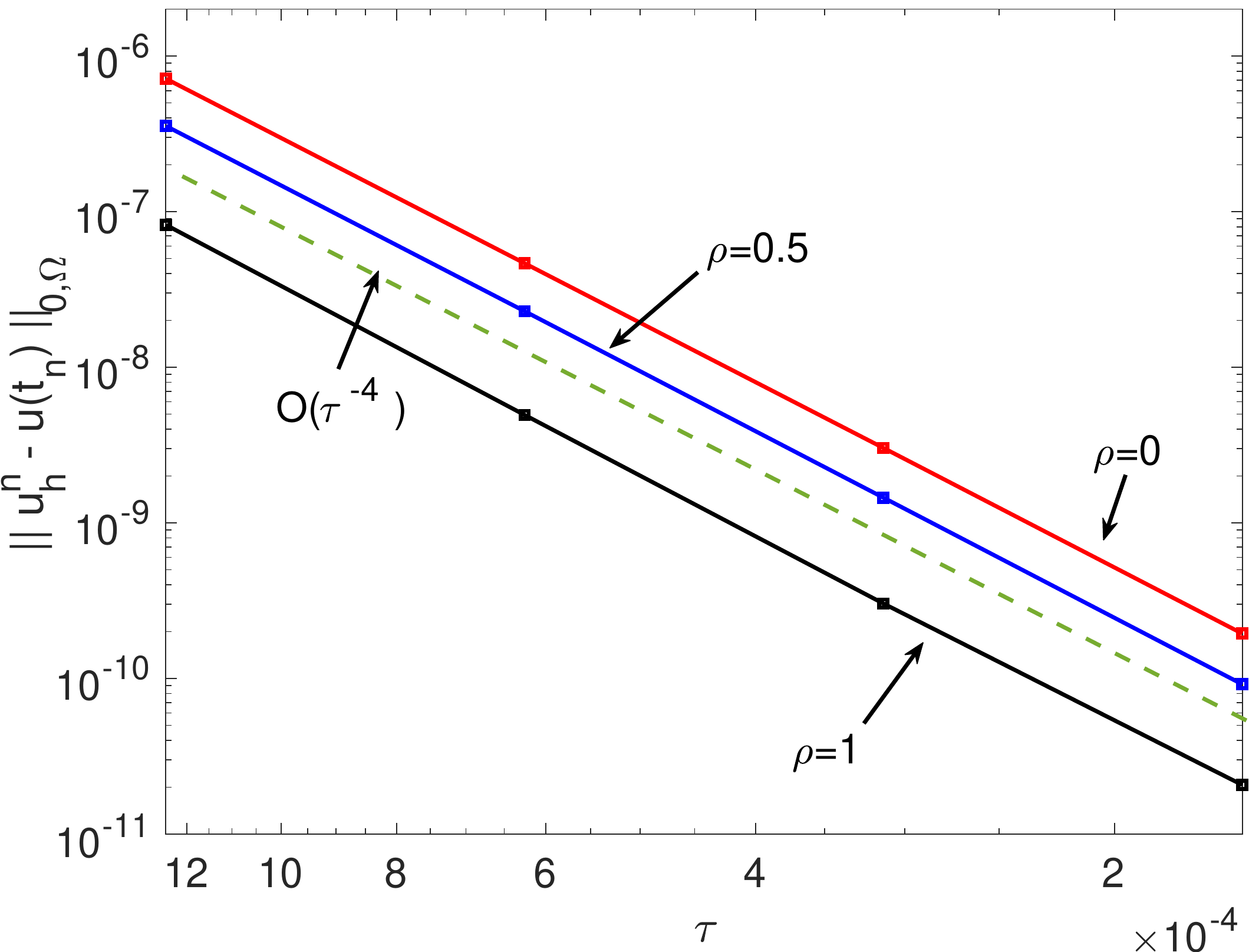}
		\includegraphics[width=0.49\textwidth]{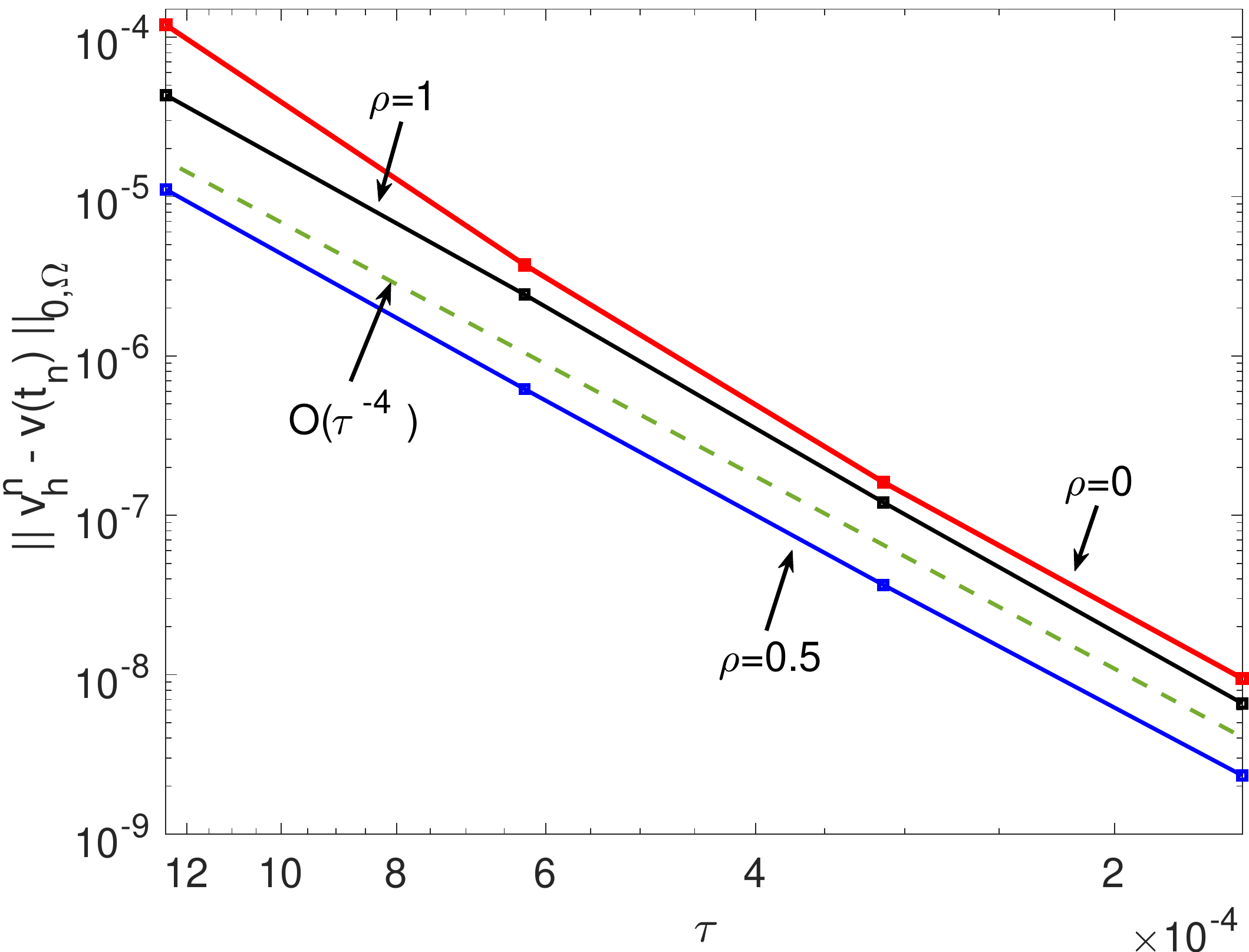}
		\caption{$L^2(\Omega)$ norm error at $T=0.1$ for different values of user-defined parameters.}
		\label{fig:l2_err_tau} 
	\end{figure}
	
	\subsection{Dispersion behaviour of the generalized-$\alpha$ method}\label{sec:dis}
	
	In this example, we numerically show that choosing higher-order methods leads to better approximations in low-frequency zones. Again, we consider a homogeneous Dirichlet boundary condition with an exact solution given by
	\begin{equation}
	u(x,t)=\sin(j\pi x)\cos(\pi t),
	\end{equation}
	and setting $\omega=\frac{1}{j}$. We discretize the spatial domain using $N=400$ elements with polynomials of degree $p=4$ and regularity $C^3$. We refer to Figure~\ref{fig:disp}, where we consider two cases; one is with the time step $\tau=0.05$ in which the bifurcation region is reached and, therefore, we obtain distinguishable results by setting different values for $\rho$. Additionally, it shows the importance of picking $\rho$ wisely. In the other example, we set the time step $\tau=10^{-3}$. Here, the spectral behaviour of the system does not change by varying $j$ and results in similar solutions for different values of $\rho$. 
	
	\begin{figure}
		\includegraphics[width=\textwidth]{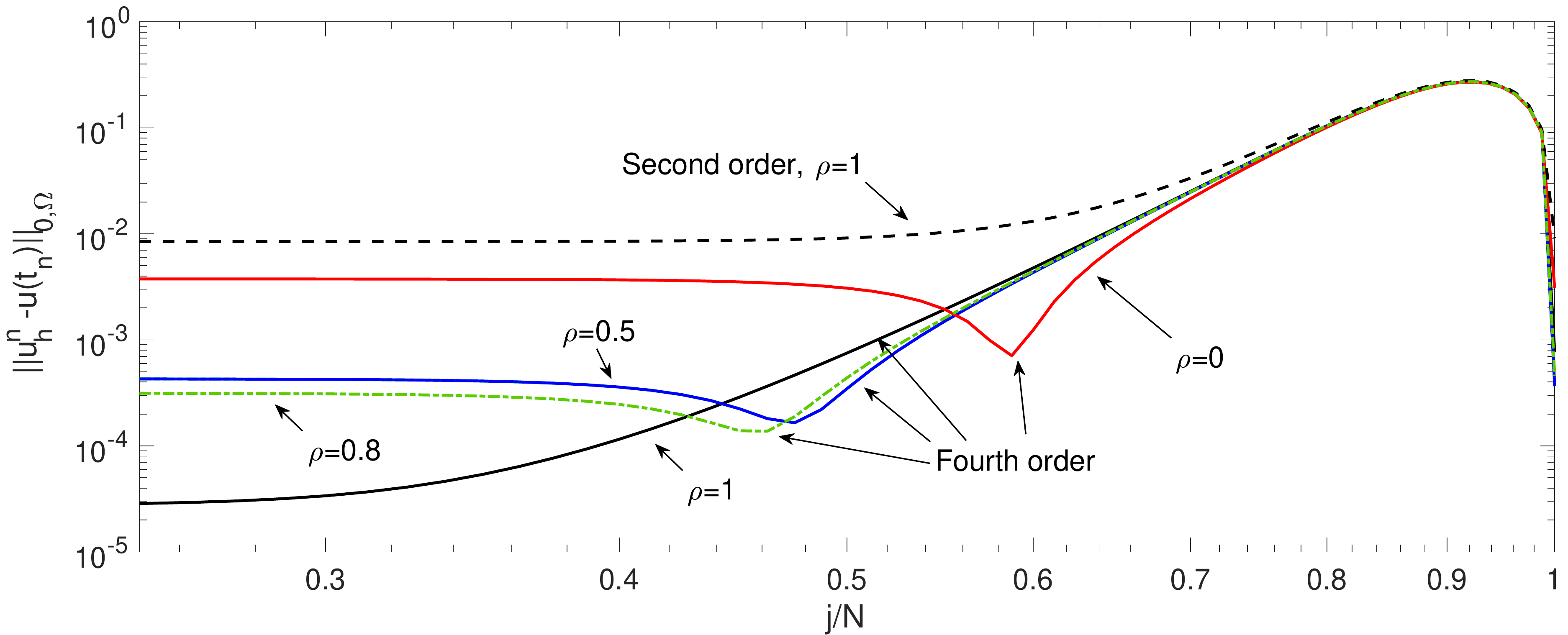}
		\includegraphics[width=\textwidth]{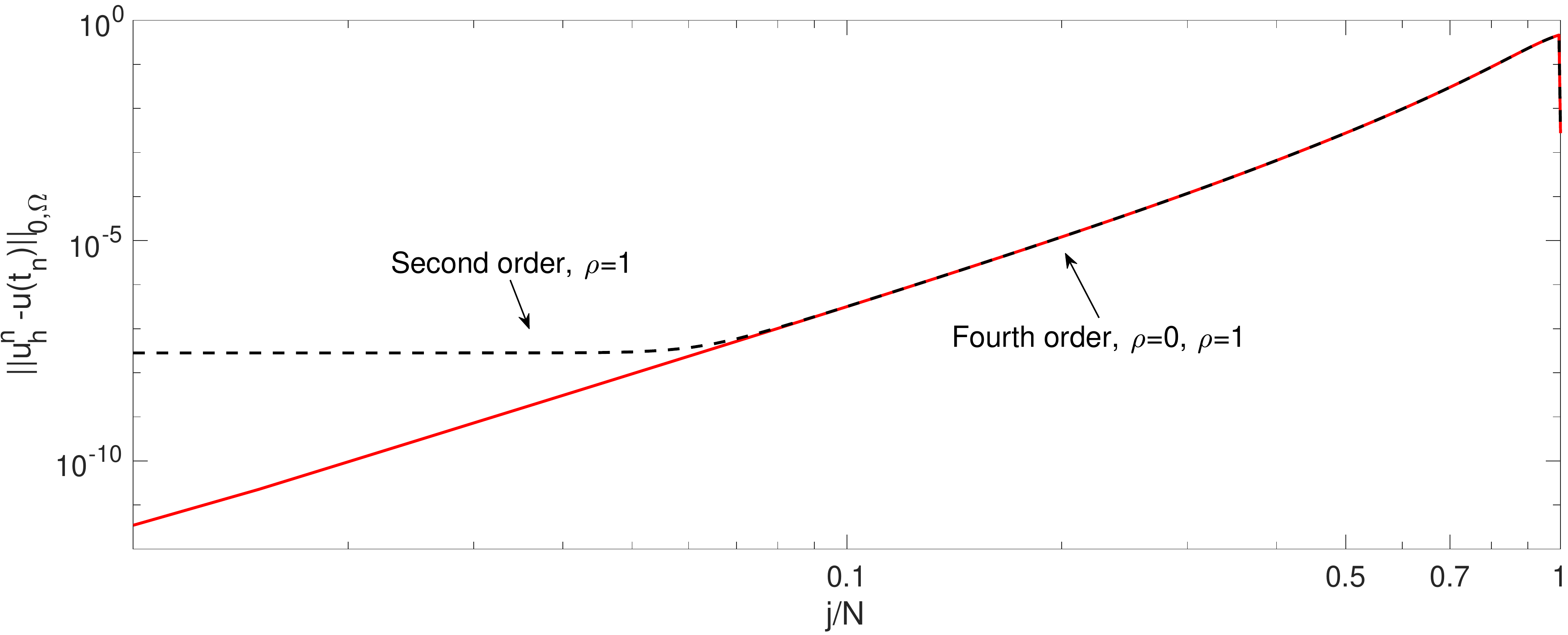}
		\caption{Comparison between the dispersion behaviour of the second and fourth order methods, $\tau=0.05$ (top) and $\tau=10^{-3}$ (bottom) for final time $t_n=5$.}
		\label{fig:disp}
	\end{figure}
	
	\subsection{Preconditioner Performance}
	
	In order to analyze the behaviour of the proposed preconditioners and the proposed fourth-order method, we solve~\eqref{eq:hyp_problem} on a regular single-patch domain (Figure~\ref{fig:blade}), a singular one (Figure~\ref{fig:donut}) and a multi-patch domain (Figure~\ref{fig:fan}), obtained by glueing together seven blade-shaped patches like the one represented in Figure~\ref{fig:blade}. For each of them, source term, Dirichlet boundary condition and initial condition are chosen such that the exact solution is always
	\begin{align*}
	u(x_1,x_2,x_3,t)=\sin (x_1) \sin (x_2) \sin (x_3) \left[ \cos(20 \pi t) + \sin(20 \pi t) \right].% , & & t \in (0,T).
	\end{align*}
	
	\begin{figure}
		\begin{subfigure}{.55\textwidth}
			\includegraphics[width=\textwidth]{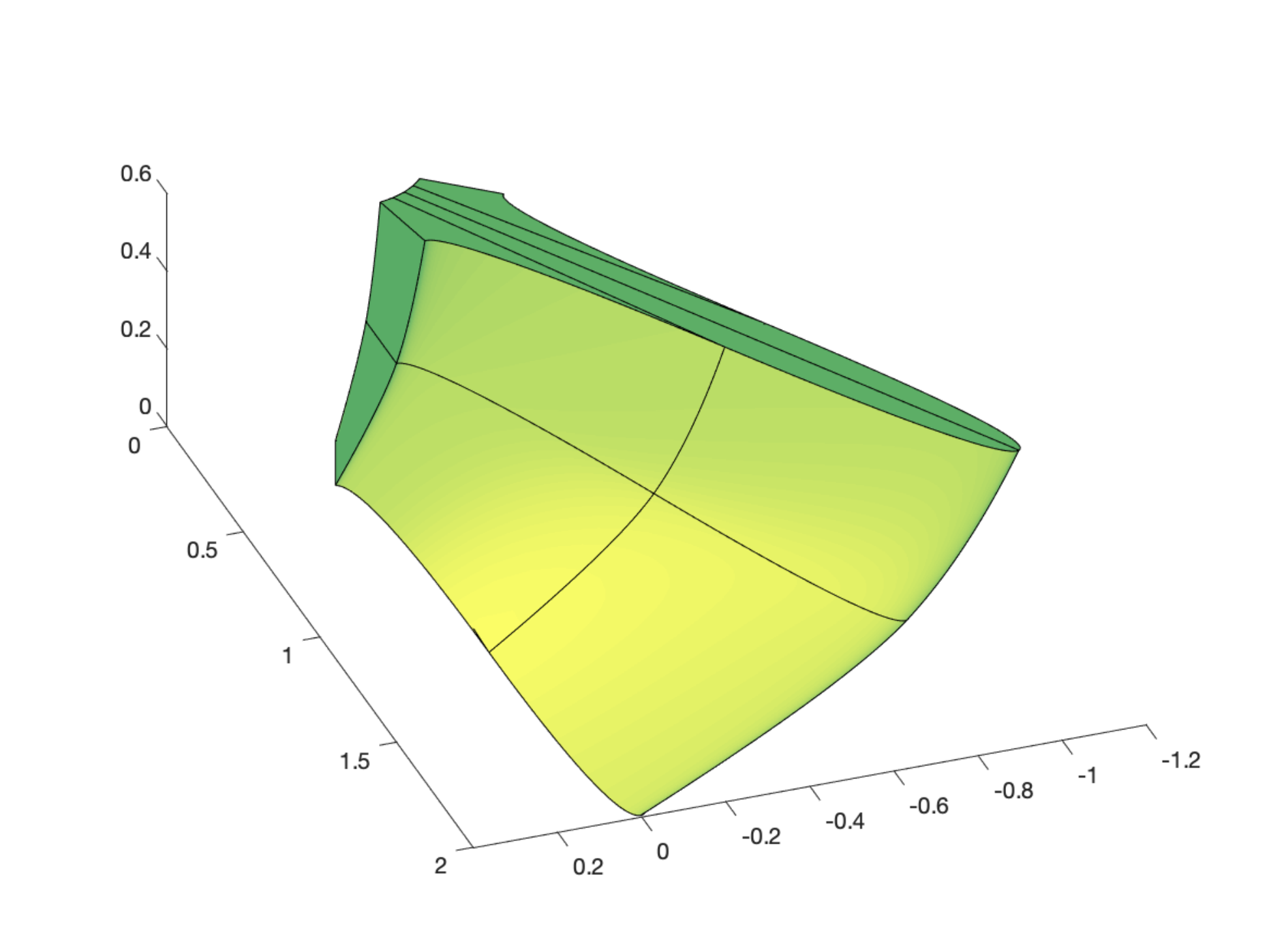}
			\caption{single-patch domain (blade geometry)}
			\label{fig:blade}
		\end{subfigure}
		\begin{subfigure}{.5\textwidth}
			\includegraphics[width=\textwidth]{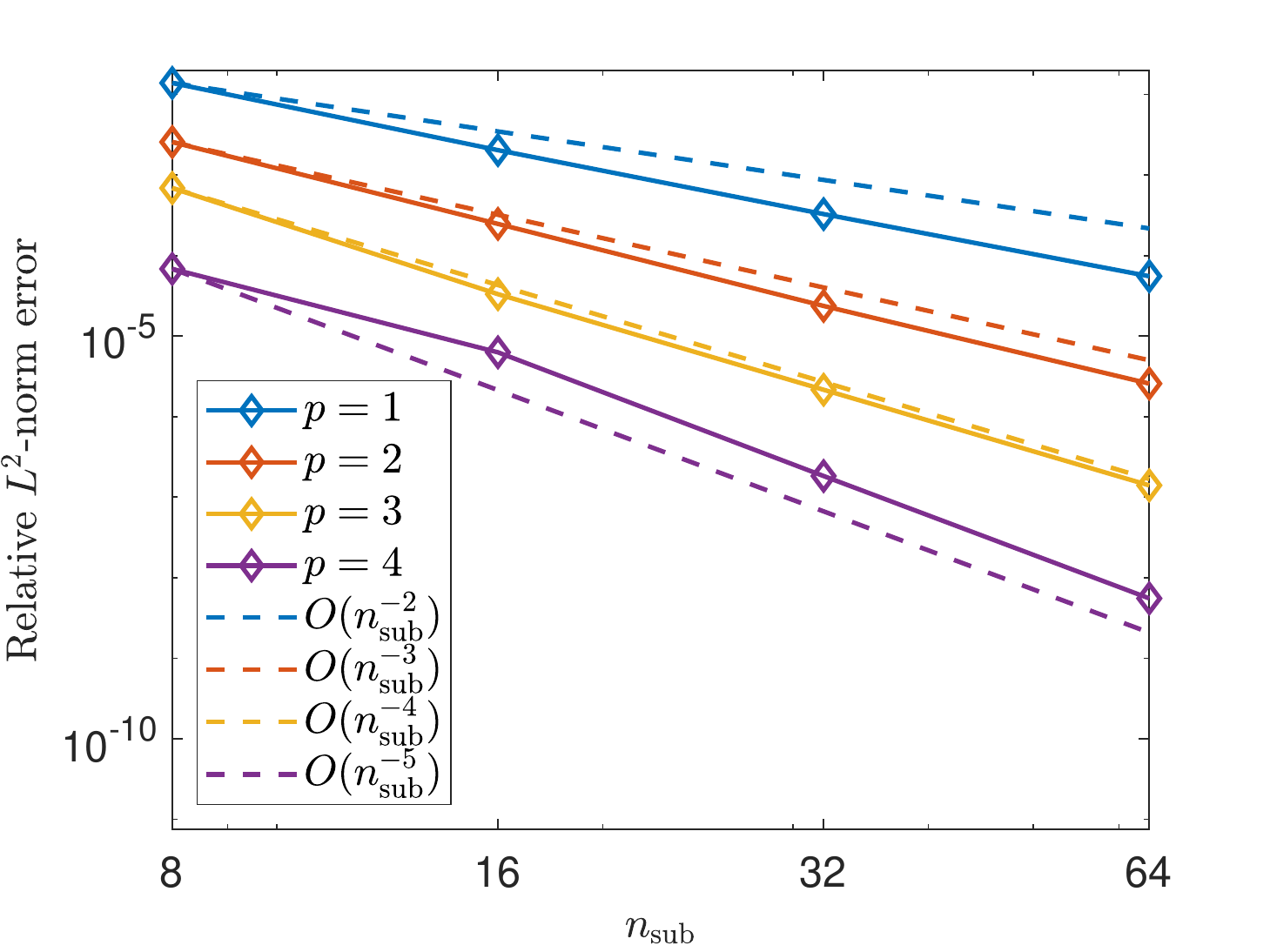}
			\caption{Optimal convergence of $L_2$ error}
			\label{fig:bladeerr}
		\end{subfigure}
		\caption{$L^2(\Omega)$ norm relative error at $T=64 \cdot \tau$ with $\tau=10^{-5}$ for the blade geometry.}
	\end{figure}
	\begin{figure}
		\begin{subfigure}{.55\textwidth}
			\includegraphics[width=\textwidth]{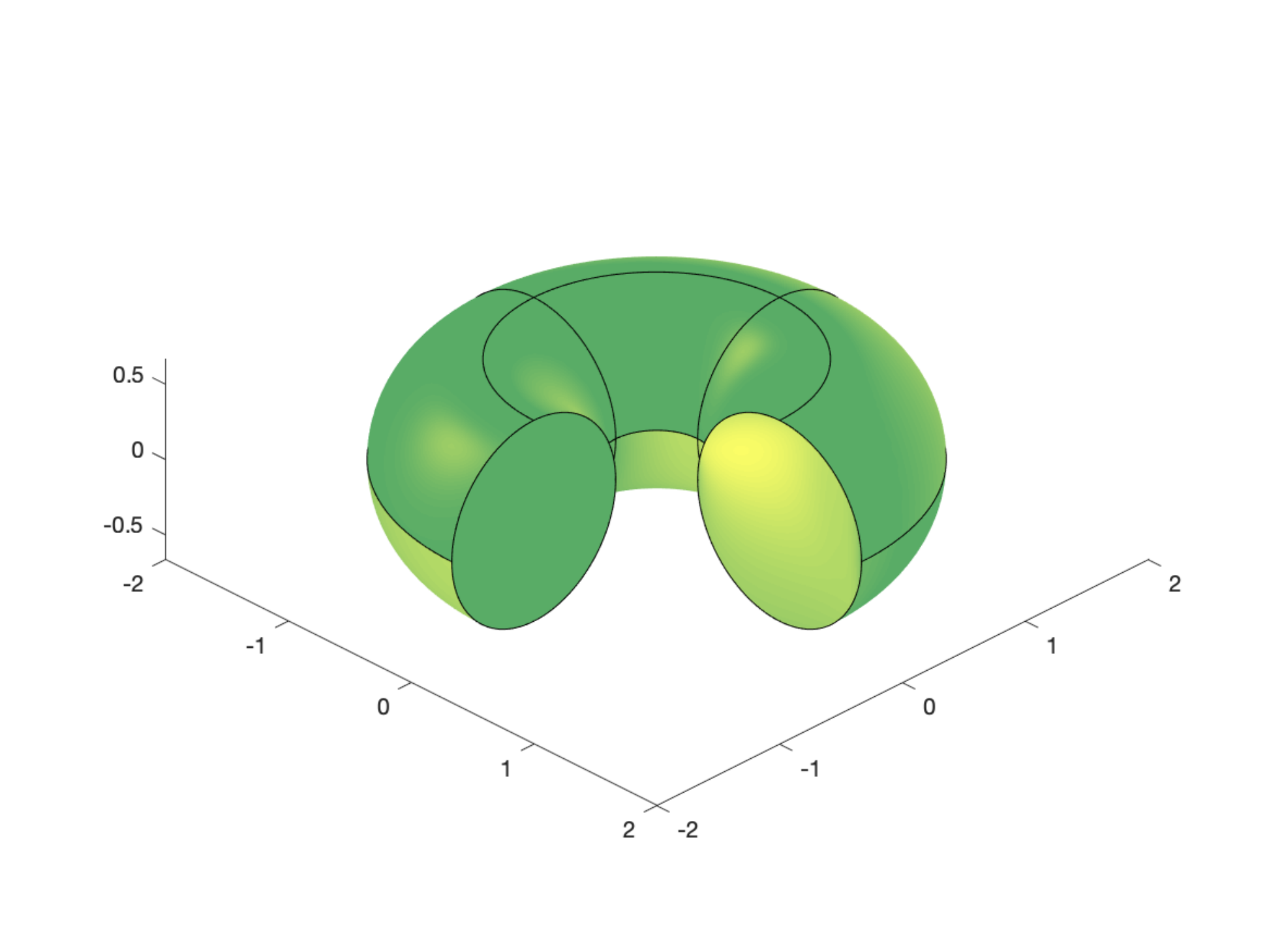}
			\caption{Singular domain (donut geometry)}
			\label{fig:donut}
		\end{subfigure}
		\begin{subfigure}{.5\textwidth}
			\includegraphics[width=\textwidth]{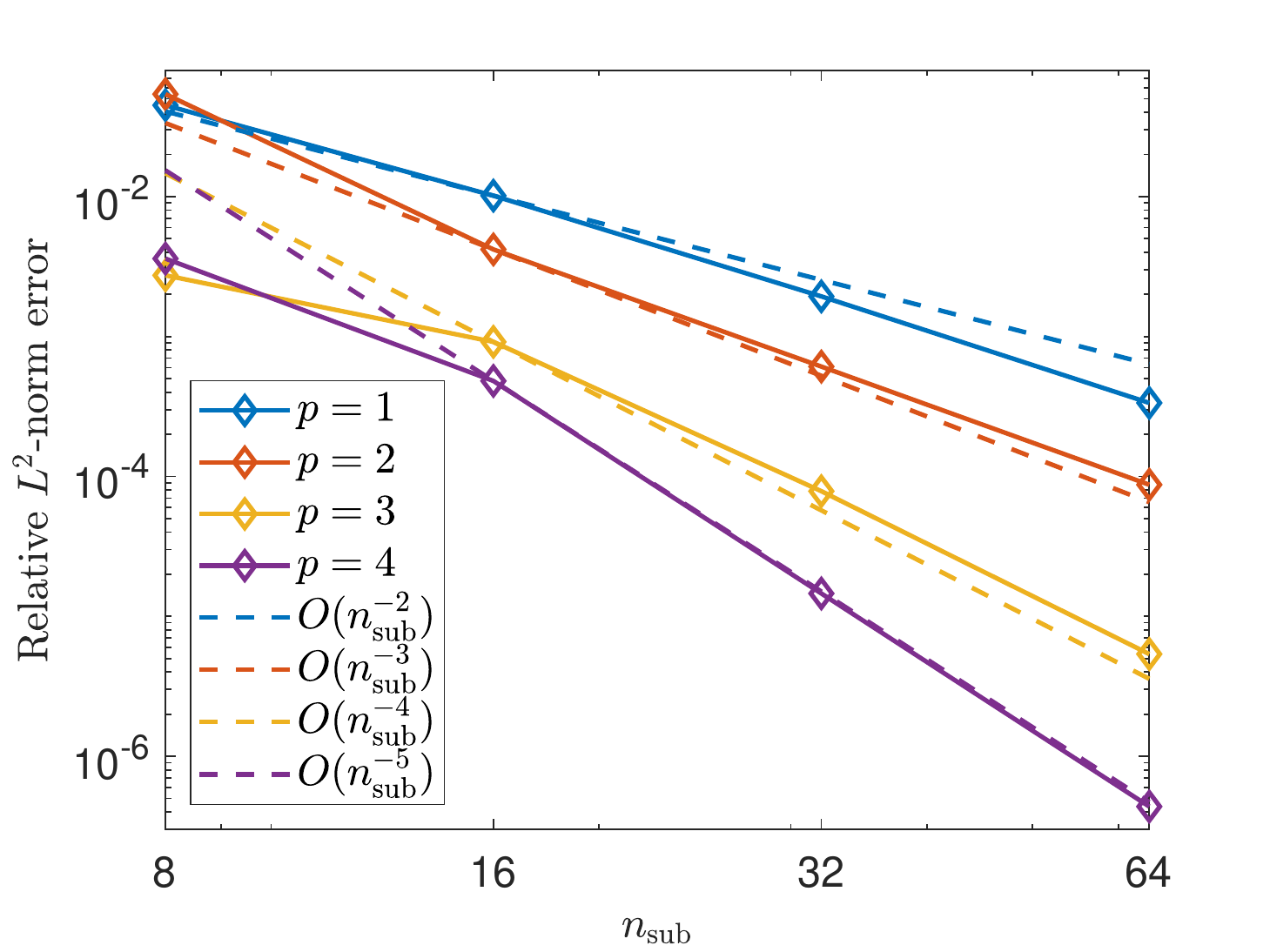}
			\caption{Optimal convergence of $L_2$ error}
			\label{fig:donuterr}
		\end{subfigure}
		\caption{$L^2(\Omega)$ norm relative error at $T=64 \cdot \tau$ with $\tau=10^{-5}$ for the donut geometry.}
	\end{figure}
	
	\begin{figure}
		\begin{subfigure}{.55\textwidth}
			\includegraphics[width=\textwidth]{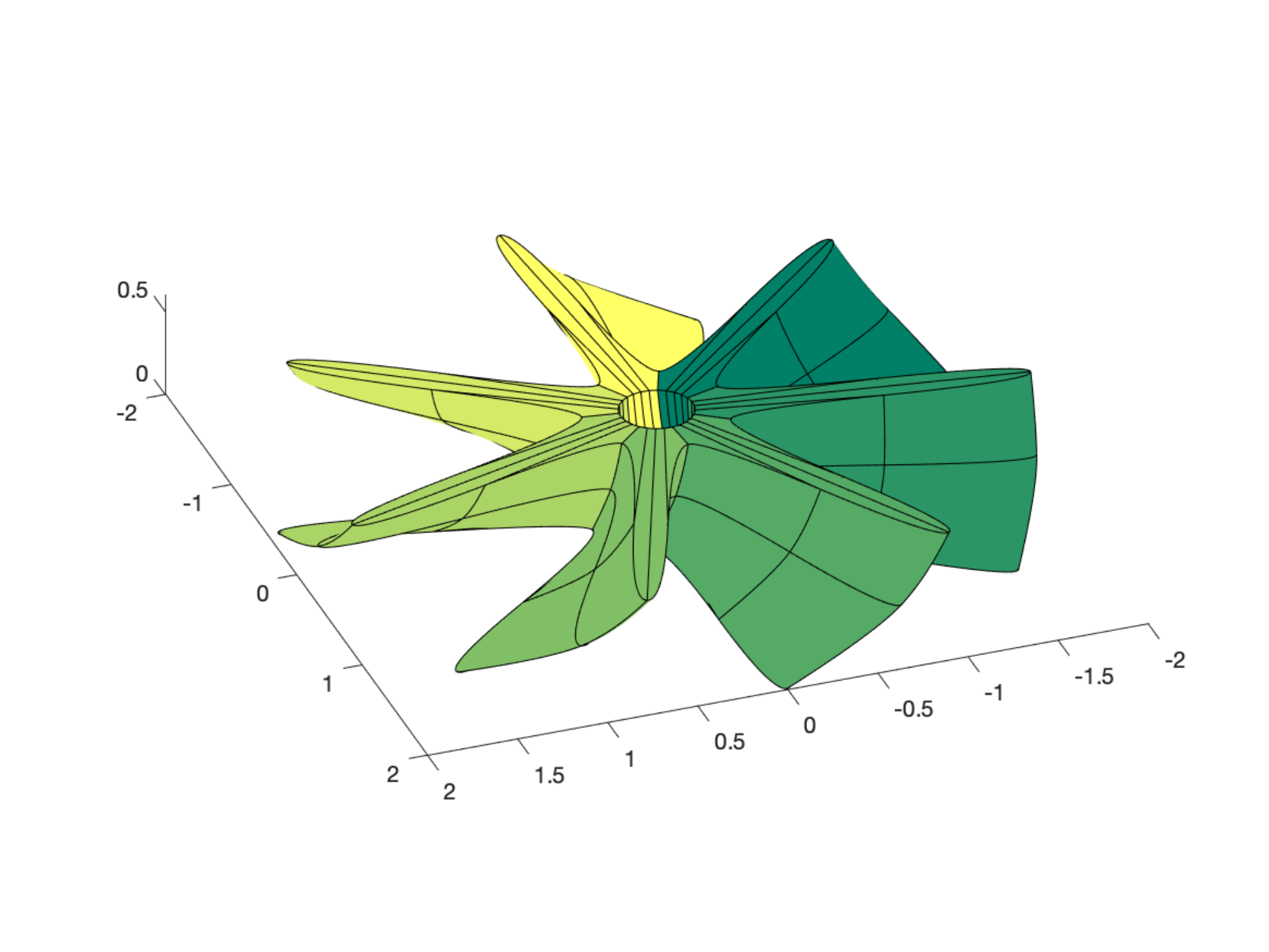}
			\caption{multi-patch domain (fan geometry)}
			\label{fig:fan}
		\end{subfigure}
		\begin{subfigure}{.5\textwidth}
			\includegraphics[width=\textwidth]{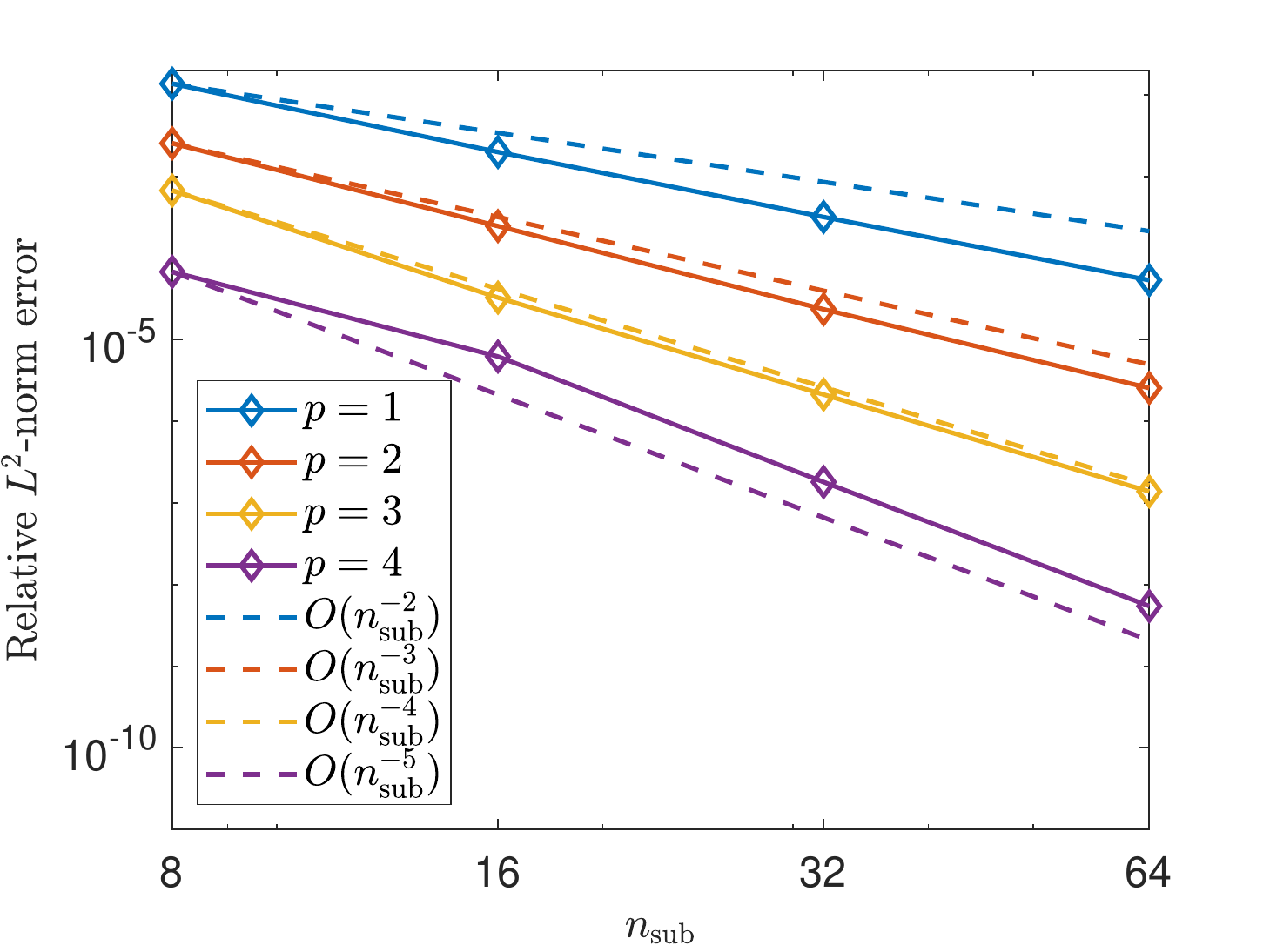}
			\caption{Optimal convergence of $L_2$ error}
			\label{fig:fanerr}
		\end{subfigure}
		\caption{$L^2(\Omega)$ norm relative error at $T=64 \cdot \tau$ with $\tau=10^{-5}$ for the fan geometry.}
	\end{figure}
	
	We report the mean value, across all the time steps, of the number of iterations needed by PCG for reaching the given tolerance for the three different spatial domains. Tables~\ref{table:Blade_its} and~\ref{table:Fan_its} shows that the number of PCG iterations is always very low and decreases when the subdivisions are increased. This is true also for the singularly parametrized Donut domain (see Table~\ref{table:Donut_its}), even though this case is beyond the robustness result presented in Section~\ref{sec:single_patch}.
	\begin{table}[]                                 
		\centering   
		\begin{tabular}{|c|c|c|c|c|c|}  
			\hline          
			$n_{\mathrm{sub}}$ & $p=1$ & $p=2$ & $p=3$ & $p=4$  \\ 
			\hline                            
			8 & 7.5 & 10.1 & 10.2 & 10.4 \\
			\hline                                                           
			16 & 6.1 & 8.4 & 8.4 & 8.9 \\
			\hline                                                           
			32 & 5.5 & 6.5 & 6.5 & 7.2 \\
			\hline                                                           
			64 & 5.0 & 5.5 & 5.5 & 6.0 \\
			\hline                                                           
		\end{tabular}
		\caption{Mean values, across all the time steps, of the iterations needed by PCG on the blade, for $\tau=10^{-5}$ and $T=64 \cdot \tau$.} 
		\label{table:Blade_its} 
	\end{table}
	
	\begin{table}[]                                 
		\centering   
		\begin{tabular}{|c|c|c|c|c|c|}  
			\hline          
			$n_{\mathrm{sub}}$ & $p=1$ & $p=2$ & $p=3$ & $p=4$  \\ 
			\hline                            
			8 & 6.0 & 8.0 & 9.0 & 10.0 \\
			\hline                                                           
			16 & 5.0 & 7.0 & 8.0 & 8.0 \\
			\hline                                                           
			32 & 5.0 & 6.0 & 7.0 & 8.0 \\
			\hline                                                           
			64 & 5.3 & 6.5 & 7.0 & 7.8 \\
			\hline                                                           
		\end{tabular} 
		\caption{Mean values, across all the time steps, of the iterations needed by PCG on the donut, for $\tau=10^{-5}$ and $T=64 \cdot \tau$.} 
		\label{table:Donut_its} 
	\end{table}   
	
	\begin{table}[]                                 
		\centering   
		\begin{tabular}{|c|c|c|c|c|c|}  
			\hline          
			$n_{\mathrm{sub}}$ & $p=1$ & $p=2$ & $p=3$ & $p=4$  \\ 
			\hline                            
			8 & 22.5 & 32.0 & 37.7 & 41.9 \\
			\hline                                                           
			16 & 21.0 & 27.8 & 30.8 & 35.6 \\
			\hline                                                           
			32 & 19.0 & 25.5 & 27.9 & 29.9 \\
			\hline                                                          
			64 & 17.0 & 22.5 & 26.0 & 27.5 \\
			\hline                                                           
		\end{tabular} 
		\caption{Mean values, across all the time steps, of the iterations needed by PCG on the fan problem, for $\tau=5 \cdot 10^{-5}$ and $T=64 \cdot \tau$.} 
		\label{table:Fan_its} 
	\end{table}   
	
	Furthermore, for assessing the good behaviour of the isogeometric discretization, for each spatial domain, we report the relative error in $L^2(\Omega)$ at the final instant $T=6.4 \cdot 10^{-4}$ with $\tau=10^{-5}$ for different mesh size and spline degree. In Figure~\ref{fig:bladeerr}, we can see that the rates of convergence are optimal with respect to the mesh size $h\approx n_{\mathrm{sub}}^{-1}$, i.e., of order $O(h^{p+1})$, for $p=1,2,3,4$, as expected from standard a priori error estimate. This optimal convergence in the spatial domain is also true for singular and multi-patch domains. For this, we refer the reader to Figures~\ref{fig:donuterr} and~\ref{fig:fanerr}.  
	
	\section{Contributions}\label{sec:7}
	
	We propose a new class of higher-order explicit generalized-$\alpha$ methods for solving hyperbolic problems that provide dissipation control. Additionally, the method's stability region is independent of the accuracy in time. We obtain $2k^{th}$-order of accuracy by solving $k$ mass-matrix systems. We adopt a preconditioner designed and built for the isogeometric mass matrix, which significantly reduces the computational costs. We discuss several numerical examples that show the stability and performance of our one-parameter family of explicit time-marching methods.

	\section*{Acknowledgement}
	This publication was also made possible in part by the CSIRO Professorial Chair in Computational Geoscience at Curtin University and the Deep Earth Imaging Enterprise Future Science Platforms of the Commonwealth Scientific Industrial Research Organisation, CSIRO, of Australia. This project has received funding from the European Union's Horizon 2020 research and innovation programme under the Marie Sklodowska-Curie grant agreement No 777778 (MATHROCKS). The Curtin Corrosion Centre and the Curtin Institute for Computation kindly provide ongoing support. The Australian Government Research Training Program Scholarship supported P. Behnoudfar's research.  Part of this work was carried over while P. Behnoudfar was invited by Prof. A. Reali in Pavia, partially supported by MIUR-PRIN project XFAST-SIMS (no. 20173C478N). G. Loli and G. Sangalli were partially supported by the European Research Council through the FP7 Ideas Consolidator Grant HIGEOM n.616563, and by the Italian Ministry of Education, University and Research (MIUR) through the "Dipartimenti di Eccellenza Program (2018-2022) - Dept. of Mathematics, University of Pavia". They are also members of the Gruppo Nazionale Calcolo Scientifico - Istituto Nazionale di Alta Matematica (GNCS-INDAM). These supports are gratefully acknowledged.

	\bibliography{ref}
\end{document}